\documentclass[11pt]{article}%
\usepackage{amssymb,amsmath,amsfonts,amsthm,array,bm,bbm,color,tikz}%
\usepackage{graphicx}
\usepackage{hyperref}
\hypersetup{colorlinks=true, linkcolor=red, anchorcolor=blue, citecolor=blue}
\usepackage{graphicx}

\usepackage{tikz}

\providecommand{\U}[1]{\protect\rule{.1in}{.1in}}

\numberwithin{equation}{section}

\newtheorem{theorem}{Theorem}[section]
\newtheorem{lemma}[theorem]{Lemma}

\newtheorem{proposition}[theorem]{Proposition}
\newtheorem{remark}[theorem]{Remark}

\newtheorem{definition}[theorem]{Definition}
\newtheorem{assumption}[theorem]{Assumption}

\def\<{\langle}
\def\>{\rangle}
\def\d{\,{\rm d}}
\def\&{\,&}

\def\div{{\rm div}}
\def\supp{{\rm supp}\,}

\def\N{\mathbb{N}}

\def\R{\mathbb{R}}

\def\B{\mathcal{B}}
\def\p{\partial}

\def\eps{\varepsilon}
\def\vphi{\varphi}
\def\I{\mathcal{I}}

\def\H{\mathcal{H}}

\begin{document}

\title{From Fisher information decay for the Kac model to the Landau-Coulomb hierarchy}

\author{José Antonio Carrillo\footnote{Email: carrillo@maths.ox.ac.uk. Mathematical Institute, University of Oxford, Oxford, OX2 6GG, UK.} \,   Shuchen Guo\footnote{Email: guo@maths.ox.ac.uk. Mathematical Institute, University of Oxford, Oxford, OX2 6GG, UK.} }

\maketitle

\vspace{-20pt}

\begin{abstract}
We consider the Kac model for the space-homogeneous Landau equation with the Coulomb potential. We show that the Fisher information of the Liouville equation for the unmodified $N$-particle system is monotonically decreasing in time. The monotonicity ensures the compactness to derive a weak solution of the Landau hierarchy. 
\end{abstract}



\section{Introduction}\label{introduction}

Landau in his seminal work \cite{landau1936kinetische} in the 1930s derived a nonlinear nonlocal diffusion equation, now bearing his name, starting from the Boltzmann equation by the so-called \textit{grazing collision limit}. While this equation is fundamental in understanding weak collision phenomena in plasma physics, its mathematical derivation and the properties of its solutions have not been elucidated even in the spatial homogeneous case till very recently.
Here, we study the derivation of the 3D space-homogeneous Landau equation with the Coulomb potential, Landau-Coulomb equation, that reads as follows
\begin{equation}\label{eq Landau}
\p_t f(v^1)=\div_{v^1} \left[ \int_{\R^3} A(v^1-v^2)\left( \nabla_{v^1}f(v^1)f(v^2)-f(v^1)\nabla_{v^2}f(v^2)\right) \d v^2\right],\quad f(0,v^1)=f^0(v^1),
\end{equation}
where the matrix-valued function $A(z)$ is the product of the the potential $a(z)$ and the projection matrix $\Pi(z)$ defined as follows
$$
A(z):=a(z)\Pi(z),\quad\text{with}\quad a(z):=\frac{1}{\left|z\right|}\quad\text{and}\quad \Pi(z):=\operatorname{Id}-\frac{z \otimes z}{|z|^2}.
$$
The Landau equation can be also reformulated into the form
$$
\p_tf(v^1) =\div_{v^1} \left[\big(A\ast f\big)(v^1)\nabla_{v^1} f(v^1)- \big(B\ast f\big)(v^1)\,f(v^1)\right],
$$
with the vector $B$ is further defined as
$$
B(z)=\div A(z)=-\frac{2z}{|z|^3}.
$$

The solution $f(t,v^1)$ represents the density of particles in a plasma with velocity $v^1\in \R^3$ at time $t$. By definition, the one-particle marginal satisfies $f_1(t,v^1)=f(t,v^1)$; if we assume statistical independence, then the two-particle marginal factorises as  $f_2(t,v^1,v^2)=f(t,v^1)f(t,v^2)$. Under this assumption, $f_1$ satisfies the equation (hierarchy)
$$
\p_t f_{1}(v^1)=\div_{v^1}\left[\int_{\R^3}A(v^1-v^{2})\left(\nabla_{v^1}f_{2}(v^1,v^2)-\nabla_{v^2}f_{2}(v^1,v^2)\right)\d v^2\right].
$$
More generally, for any $m\in\N$, the equation governing the evolution of the $m$-particle marginals is given by the Landau hierarchy on $\R^{3m}$,
\begin{equation}\label{eq Landau hierarchy}
\p_t f_{m}=\sum_{i=1}^m\div_{v^i}\left[\int_{\R^3}A(v^i-v^{m+1})\left(\nabla_{v^1}f_{m+1}-\nabla_{v^{m+1}}f_{m+1}\right)\d v^{m+1}\right],
\end{equation}
with initial data $f_m(0)=(f^0)^{\otimes m}$. An observation is that if the one-particle distribution  $f$  satisfies the Landau equation \eqref{eq Landau}, then the tensorised solution  $f_m=(f)^{\otimes m}$  solves the hierarchy \eqref{eq Landau hierarchy}. 

However, even if we assume that the initial data $f_m(0)$ factorises, the tensorisation property is immediately destroyed due the the interaction between particles. To better understand this property, Kac proposed a framework of \textit{propagation of chaos} in the seminal work \cite{kac1956foundations}, which states that if particles are initially independent, they remain asymptotically independent under the evolution. Proving  propagation of chaos for the Landau-Coulomb equation \eqref{eq Landau}, however, is a longstanding challenge. To answer this problem in the \textit{mean-field} scaling regime, we would like to divide it into two key steps. The first one is to derive the Landau-Coulomb hierarchy \eqref{eq Landau hierarchy} from an underlying particle system; the second step is to show that the hierarchy allows only the tensorised solution, namely, to prove the uniqueness of the Landau-Coulomb hierarchy \eqref{eq Landau hierarchy}. We also refer to \cite{bresch2022new} by Bresch-Jabin-Soler for a similar strategy used for the Vlasov-Fokker-Planck equation. In this work, we provide a rigorous proof of the first step for the Landau-Coulomb hierarchy \eqref{eq Landau hierarchy}. 

In \cite{kac1956foundations}, Kac considered the $N$-particle model given by Markov processes, and proved the large $N$ limit towards the space-homogeneous Boltzmann equation in a simpler case. To be more precise, the evolution of the exchangeable particle system is governed by the Liouville equation $\p_t f_N=\mathcal{L}_N f_N$, and the first marginal of $f_N$ converges to the solution of the Boltzmann equation. We refer to the work by Mischler-Mouhot \cite{mischler2013kac} for more discussion on Kac model for the Boltzmann equation and its mean-field limit. One of the main features of Kac's particle model is that the collision preserves the momentum and energy. 

Villani \cite{villani1998new} gives the first rigorous proof of the grazing collision limit from the Boltzmann equation to the Landau equation that shows the existence of H-solutions for both by providing a weak compactness to the Coulomb singularity. Following the idea presented in \cite{villani1998new}, one can construct the Kac model and the corresponding Liouville equation for the space-homogeneous Landau equation; see Miot-Pulvirenti-Saffirio \cite{Miot2011kac} and Carrapatoso \cite{carrapatoso2016propagation}. The Liouville equation describes the evolution of $N$-particle joint distribution with initial value $f_N(0)=(f^0)^{\otimes N}$ on $\R^{3N}$, which reads as
\begin{equation}\label{eq fN}
\p_t f_{N}=\frac{1}{N}\sum_{i\neq j}A(v^i-v^j):(\nabla_{v^iv^i}^2f_N-\nabla_{v^iv^j}^2f_N)+\frac{1}{N}\sum_{i\neq j}B(v^i-v^j)\cdot (\nabla_{v^i}-\nabla_{v^j})f_N,
\end{equation}
or in a more compact form  
$$
\p_t f_{N}=\frac{1}{2N}\sum_{i\neq j}\div_{v^i-v^j}\left[A(v^i-v^j)\nabla_{v^i-v^j}f_{N}\right],
$$
where we used '$:$' as the Frobenius inner product for matrices, and used the notations
$$
\nabla_{v^i-v^j}g_{N}(V_N):=\nabla_{v^i}g_{N}(V_N)-\nabla_{v^j}g_{N}(V_N),\quad \div_{v^i-v^j}G_{N}(V_N):=\div_{v^i}G_{N}(V_N)-\div_{v^j}G_{N}(V_N)
$$ with some function $g_N$ or vector field $G_N$ on $V_N=(v^1,\ldots,v^N)\in \R^{3N}$. Note that the joint distribution $f_N$ is symmetric, namely, for $\ell\neq k$, it has
\begin{equation}\label{symmetry}
f_N(\cdots,v^k,\cdots,v^\ell,\cdots)=f_N(\cdots,v^\ell,\cdots,v^k,\cdots).    
\end{equation}
The formulation of the Liouville equation for the Kac model \eqref{eq fN}, the Landau master equation, actually dates back to Prigogine and Balescu in 1950s \cite{prigogine1959irreversibleI,prigogine1959irreversibleII}. Sometimes it also refers as Balescu-Prigogine master equation; see also \cite{kiessling2004master}. 

In fact, one can construct a particle system associated formally to the Liouville equation for the Kac model \eqref{eq fN}. The following stochastic $N$-particle systems is proposed by Carrapatoso in \cite{carrapatoso2016propagation} and has also been considered in \cite{fournier2017kac}. It reads as  
\begin{equation}\label{kacparticles}
\d V_t^i=\frac{2}{N}\sum_{j\neq i}^{N}B(V_t^i-V_t^j)\d t+\sqrt{\frac{2}{N}}\sum_{j\neq i}^{N}A^{\frac{1}{2}}(V_t^i-V_t^j)\d Z_t^{i,j},
\end{equation}
where for $1\leq i<j\leq N$,  $Z_t^{i,j}=W_t^{i,j}$ are $N(N-1)/2$   i.i.d.  $3$-dimensional Brownian motions, while $Z_t^{j,i}=-W_t^{i,j}$ are anti-symmetric, and the diffusion coefficient matrix is
the unique square root of the nonnegative symmetric matrix.  We also assume that the initial data $(V_0^1,\ldots,V_0^N)$ are i.i.d with the distribution $f^0$. Under above assumptions on the particle system, the particles are exchangeable, and conserve the total momentum and kinetic energy almost surely, namely, 
$$
    \sum_{i=1}^NV_t^i=\sum_{i=1}^NV_0^i,\quad\text{and}\quad\sum_{i=1}^N|V_t^i|^2=\sum_{i=1}^N|V_0^i|^2.
$$

It is worth mentioning that the Landau equation with a Coulomb potential is considered as the most physically significant case. However, it is particularly hard to be derived from particle systems and their associated Liouville equations due to the singularity and degeneracy of the matrix $A(z)$. Actually, it is another challenging open problem to discuss the well-posedness of the SDE system \eqref{kacparticles} for the Coulomb singularity.

In \cite{Miot2011kac}, Miot-Pulvirenti-Saffirio showed that the marginals of a modified Liouville equation, where the matrix $A(z)$ bounded from below,  weakly converge in sense of probability measures to a solution of Landau-Coulomb hierarchy \eqref{eq fN}. Also, it suggests that with a nicer matrix $A(z)$, some better results can be obtained. Carrapatoso in \cite{carrapatoso2016propagation} also showed quantitative and uniform-in-time propagation of chaos for the Landau equation with Maxwellian molecules, corresponding to the case where $A(z)=|z|^2\Pi(z)$. Finally, Fournier-Guillin proved the quantitative propagation of chaos for the Landau equation with hard potentials where they assume $A(z)=|z|^{\gamma+2}\Pi(z)$ with $\gamma\in(0,1]$;  see \cite{fournier2017kac}.

Apart from the Kac model for particles \eqref{kacparticles} with the formally associated Liouville equation given by \eqref{eq fN}, there is another formulation of the particle system. The SDE representation is given by  
$$
\d \hat{V}_t^i=\frac{2}{N}\sum_{j=1}^{N}B(\hat{V}_t^i-\hat{V}_t^j)\d t+\sqrt{2}\Big(\frac{1}{N}\sum_{j=1}^{N}A(\hat{V}_t^i-\hat{V}_t^j)\Big)^{\frac{1}{2}}\d W_t^i,
$$
with its Liouville equation
\begin{equation}\label{alternative}
\p_t \hat{f}_N= \sum_{i=1}^{N}\div_{v^{i}}\Big[\frac{1}{N}\sum_{j=1}^{N}A(v^{i}-v^{j})\nabla_{v^{i}}\hat{f}_N-\frac{1}{N}\sum_{j=1}^{N}B(v^{i}-v^{j})\hat{f}_N\Big].    
\end{equation}
Notice that the particle system does not conserve almost surely the total momentum and energy, only the expectation of total momentum and energy are conserved. The particle system has also been considered in the mean-field derivation of the Landau equation. For example, the first quantitative propagation of chaos result for the Landau equation with Maxwellian molecules is obtained by Fontbona-Guérin-Méléard in \cite{fontbona2009measurability} using optimal transport techniques, which is improved in \cite{fournier2009particle} with a better rate by the coupling method. Fournier-Hauray derived the Landau equation with moderately soft potentials where $A(z)=|z|^{\gamma+2}\Pi(z)$ with $\gamma\in(-2,0)$; they obtain the convergence rate for $\gamma\in(-1,0)$ by a refined coupling\textit{} method, and the propagation of chaos without rate for $\gamma\in(-2,-1]$ by the martingale method; see \cite{fournier2016propagation}. The relative entropy method has been used to quantitatively approximate the Landau-like equation in \cite{carrillo2024mean} and  the Landau equation with Maxwellian molecules in \cite{carrillo2024relative}. 

A recent breakthrough in kinetic theory is the work of Guillen-Silvestre \cite{guillen2025landau}, who proved the monotonicity of Fisher information for the space-homogeneous Landau equation. This result played a crucial part in proving, for the first time, the global well-posedness of classical solutions to the Landau-Coulomb equation. The monotonicity theorem was later generalised to the space-homogeneous Boltzmann equation by Imbert-Silvestre-Villani \cite{imbert2024monotonicity}. These developments further highlight the fundamental role of Fisher information in kinetic theory. For a comprehensive overview, we refer to the recent review by Villani \cite{villani2025fisher}.

Our main contribution is to 
provide the first rigorous derivation of the Landau-Coulomb hierarchy \eqref{eq Landau hierarchy} from the unmodified particle system determined by the Liouville equation \eqref{eq fN} in the strong sense. This is achieved by showing that the monotonicity of Fisher information remains valid at the level of the Kac's particle system of the Landau-Coulomb equation given by the Liouville equation \eqref{eq fN}. In contrast, this property is unclear to hold for the alternative Liouville equation \eqref{alternative}. This result constitutes one of the two key steps in the full justification of the Landau-Coulomb equation \eqref{eq Landau} from a system of mean-field interacting particles.

\subsection{Main results}

Throughout the paper, we will use the notation $V_M=(v^1,\ldots,v^M)\in \R^{3M}$ for $M\in\N$. For any density function $g_M$ on $\R^{3M}$, we denote the (renormalised) entropy by $\H_M$ and the (renormalised) Fisher information by $\I_M$ as follows,
$$
\H_M(g_M):=\frac{1}{M}\int_{\R^{3M}}g_M\log g_M,
$$
and
$$
\I_M(g_M):=\frac{1}{M}\int_{\R^{3M}}g_M|\nabla\log g_M|^2 \d V_M=\frac{1}{M}\int_{\R^{3M}}\frac{|\nabla g_M|^2}{g_M}\d V_M=\frac{4}{M}\int_{\R^{3M}}|\nabla\sqrt{ g_M}|^2 \d V_M.
$$
\begin{remark}
\label{lemma subadditivity}
If $g_N$ is a probability density function on $\R^{3N}$ and its $m$-marginal $g_{N,m}$ for any $m\in\N$ is given by integrating it over the variables from $v^{m+1}$ to $v^N$,
$$
g_{N,m}=\int_{\R^{3(N-m)}}g_N\d v^{m+1}\cdots\d v^{N},
$$
then the sub-additivity holds for both entropy and Fisher information, i.e., 
$$
\H_m(g_{N,m}):=\frac{1}{m}\int_{\R^{3m}}g_{N,m}\log g_{N,m}\d V_m\leq \frac{1}{N}\int_{\R^{3N}}g_{N}\log g_{N}\d V_N=\H_N(g_N),
$$
and 
$$
\I_m(g_{N,m}):=\frac{1}{m}\int_{\R^{3m}}g_{N,m}\left|\nabla_{3m}\log g_{N,m}\right|^2\d V_m\leq \frac{1}{N}\int_{\R^{3N}}g_{N}\left|\nabla_{3N}\log g_{N}\right|^2\d V_N=\I_N(g_N).
$$
Particularly, if $g_N$ is tensorised such as $g_N=g^{\otimes N}$, then
$$
\H_1(g)=\H_N(g_N),\quad\text{and}\quad \I_1(g)=\I_N(g_N);
$$
see the proofs for instance in \cite{hauray2014kac}.
\end{remark}  
We now present the assumptions on the initial data $f^0$. 
\begin{assumption}\label{assumption}
We assume the initial data satisfies the following conditions 
\begin{itemize}
    \item [1.] Finite mass, momentum and energy: 
    $$
    \int_{\R^3}f^0=1,\quad \int_{\R^3}f^0v_\alpha=0\,\,(\mbox{for }\alpha=1,2,3), \,\,\mbox{and  } \int_{\R^3}f^0|v|^2=3.
    $$
    \item [2.] Finite entropy: $\H_1(f^0)=\int_{\R^{3}}f^0\log f^0<+\infty$.
    \item [3.] Finite Fisher information: $\I_1(f^0)=\int_{\R^{3}}f^0|\nabla\log  f^0|^2<+\infty$.
\end{itemize}
\end{assumption}

As the main object in this work, we understand solutions of the Liouville equation \eqref{eq fN} in the following sense. 
\begin{definition}[Weak solutions of the Liouville equation]\label{def fN} 
For any $T>0$, a weak solution of the Liouville equation \eqref{eq fN} with $f_{N}(0)=(f^0)^{\otimes N}$ satisfies the following
weak form, for any compactly supported $\vphi_{N}\in W^{4,\infty}(\R^{3N})$,
\begin{equation}\label{eq weak Liouville}
\begin{aligned}
\int_{\R^{3N}}\vphi_{N} f_{N}(T)\&\d V_{N}-\int_{\R^{3N}}\vphi_{N} (f^0)^{\otimes N}\d V_{N}\\=\&\frac{1}{N}\sum_{i\neq j}^N\int_0^T\int_{\R^{3N}}A(v^i-v^j):(\nabla_{v^iv^i}^2\vphi_{N}-\nabla_{v^iv^j}^2\vphi_{N})f_{N}\d V_{N}\d t\\\&+\frac{1}{N}\sum_{i\neq j}^N\int_0^T\int_{\R^{3N}}B(v^i-v^j)\cdot (\nabla_{v^i}\vphi_{N}-\nabla_{v^j}\vphi_{N}) f_{N}\d V_{N}\d t.
\end{aligned}
\end{equation}
Additionally, it conserves mass and energy, and its entropy and Fisher information are monotonically decreasing, namely, for any $0\leq t_1\leq t_2\leq T$,  
\begin{equation}\label{fN conservation}\int_{\R^{3N}}f_{N}(t_1)\d V_N=\int_{\R^3}f^0(v^1)\d v^1,\quad\int_{\R^{3N}}f_{N}(t_1)|V_N|^2\d V_N=N\int_{\R^3}f^0(v^1)|v^1|^2\d v^1,\end{equation}
\begin{equation}\label{fN entropy decay}
\H_N(f_{N}(t_2))\leq \H_N(f_{N}(t_1))\leq  \H_N(f_{N}(0))=\H_1(f^0),
\end{equation}
and 
\begin{equation}\label{fN Fisher decay}
\I_N(f_{N}(t_2))\leq \I_N(f_{N}(t_1))\leq  \I_N(f_{N}(0))=\I_1(f^0).
\end{equation}
\end{definition}
In Section \ref{section uniform estimate}, we will show there exists a unique weak solution satisfying the definition above; see Proposition \ref{prop well-posedness of fN}. The assumption of higher regularity for the test function $\vphi_N$ in the definition is necessary because we will require the hessian $\nabla^2\vphi_N$ to be twice differentiable in order to pass to the limit, we refer to Proposition \ref{prop pass to the limit} for more details. The marginal distribution satisfies the so-called BBGKY hierarchy:
\begin{equation}\label{BBGKY}
\begin{aligned}
\p_t f_{N,m}
=\&\frac{1}{2N}\sum_{i\neq j\leq m}\div_{v^i-v^j}\left[A(v^i-v^j)\nabla_{v^i-v^j}f_{N,m}\right]\\\&+\frac{N-m}{N}\sum_{i=1}^m\div_{v^i}\int_{\R^3}\left[A(v^i-v^{m+1})\nabla_{v^i-v^{m+1}}f_{N,m+1}\right]\d v^{m+1}.
\end{aligned}
\end{equation}
Intuitively, as $N\to\infty$, the first term on the right-hand side of \eqref{BBGKY} is of order $O(m^2/N)$ and vanishes asymptotically, while the second term remains of order $O(1)$. This suggests that in the limit, the hierarchy should take the form of \eqref{eq Landau hierarchy}, whose weak solutions are defined as follows.
\begin{definition}[Weak solutions of the Landau hierarchy]\label{def Landau hierarchy}
For any $T>0$, a weak solution of Landau hierarchy \eqref{eq Landau hierarchy} with $f_m=(f^0)^{\otimes m}$ satisfies the following
weak form, for any compactly supported $\vphi_{m}\in W^{4,\infty}(\R^{3m})$, 
\begin{equation}\label{eq weak Landau hierarchy}
\begin{aligned}
  \int_{\R^{3m}}\vphi_{m}f_{m}(T)\&\d V_{m}-\int_{\R^{3m}}\vphi_{m} (f^0)^{\otimes m}\d V_{m}\\=\&\sum_{i=1}^m \int_0^T\int_{\R^{3(m+1)}}  A(v^i-v^{m+1}):\nabla_{v^iv^i}^2\vphi_{m}\,f_{m+1}\d V_{m+1}\d t\\
\&+2\sum_{i=1}^m\int_0^T\int_{\R^{3(m+1)}}  B(v^i-v^{m+1})\cdot \nabla_{v^{i}}\vphi_{m}\,f_{m+1}\d V_{m+1}\d t.
\end{aligned}
\end{equation} 
Additionally, it conserves mass and energy, and its entropy and Fisher information are uniformly bounded, namely, for any $t\in[0,T]$, 
$$
\int_{\R^{3m}}f_{m}(t)\d V_m=\int_{\R^3}f^0(v^1)\d v^1,\quad\int_{\R^{3m}}f_{m}(t)|V_m|^2\d V_m=m\int_{\R^3}f^0(v^1)|v^1|^2\d v^1,$$
$$
\H_m(f_{m}(t))\leq  \H_m(f_{m}(0))=\H_1(f^0),
$$
and 
$$
\I_m(f_{m}(t))\leq  \I_m(f_{m}(0))=\I_1(f^0).
$$
\end{definition}

Our main result is stated as follows.
\begin{theorem}\label{thm main}
For any fixed $m\in\N$, all elements in the adherence points set of the sequence $\{f_{N,m}\}_{N\in\N}$, which are $m$-th marginals of weak solutions of the Liouville equation 
\eqref{eq fN},  are weak solutions of the Landau hierarchy \eqref{eq Landau hierarchy}. In other words, there exists a subsequence $\{f_{N_\ell,m}\}_{N_\ell\in\N}$ converging to a weak solution $f_m$ of the Landau hierarchy \eqref{eq Landau hierarchy} when $N_\ell\to\infty$. 
\end{theorem}

For a more detailed convergence result, we refer to  Proposition \ref{prop strong convergence}.
\begin{remark}
If the uniqueness of weak solutions of \eqref{eq Landau hierarchy} can be established, then the weak solution should coincide with the tensorised solution of the Landau equation \eqref{eq Landau}. As a consequence, uniqueness would imply the propagation of chaos, meaning that any convergent subsequence in Theorem \ref{thm main} has the unique limit given by the tensorised solution. This would close the gap of full derivation of the Landau-Coulomb equation \eqref{eq Landau} such as for any time $t>0$, $\lim_{N\to\infty}f_{N,m}=f^{\otimes m}$.
\end{remark}

The rest of the paper is organised as follows. In Section \ref{section uniform estimate}, we will show the well-posedness of the Liouville equation \eqref{eq fN}. A key novelty of our argument lies in proving the decay of the Fisher information for the Liouville equation \eqref{eq fN}.  Sections \ref{sec convergence} and \ref{sec pass limit} are then dedicated to the compactness arguments required to complete the proof of Theorem \ref{thm main}.

\section{Uniform estimate of the Liouville equation}\label{section uniform estimate}
We begin this section by stating the well-posedness of the Liouville equation. As the core ingredient of the proof, we will show that the entropy and Fisher information are monotonically decreasing along the Liouville equation.

\begin{proposition}[Well-posedness of the Liouville equation]\label{prop well-posedness of fN}
Under Assumption \ref{assumption}, there exists a unique weak solution defined in Definition \ref{def fN} of the Liouville equation \eqref{eq fN}. 
\end{proposition}

Recall the Liouville equation 
$$
\p_tf_N=\frac{1}{2N}\sum_{i\neq j}\div_{v^i-v^j}\left[A(v^i-v^j)\nabla_{v^i-v^j}f_{N}\right],
$$
where $A(z)=a(z)\Pi(z)$ with the singular potential $a(z)=1/|z|$ and the degenerate projection matrix $\Pi(z)$. To prove the existence of weak solutions of the Liouville equation  \eqref{eq fN}, it is crucial to show the entropy and Fisher information of $f_{N}^\eps$ are decreasing uniformly-in-$\eps$ along the flow given by a suitable regularised flow \eqref{eq fNeps}, and then pass to the limit in \eqref{eq fNeps}. For  notational simplicity, we first present the calculations for any  positive smooth symmetric  functions $g_N\in C^1\left([0,T],C^\infty(\R^3)\right)$ with rapid decay at infinity. We postpone the introduction of regularised system after Lemma \ref{lemma Liouviell flow} in order to highlight the main ingredients of the proof. We use the Gateaux derivative to present the functional derivative of entropy and Fisher information along the flow given by operator $Q(\cdot)$, respectively, 
$$
\<\left(\H_N(g_N)\right)',Q(g_N)\>:=\frac{\d}{\d \tau}\H_N\left(g_N+\tau Q(g_N)\right)\bigg|_{\tau=0}=\frac{1}{N}\int_{\R^{3N}}\log g_N Q(g_N)\d V_N, 
$$
and
\begin{equation}\label{def diff fisher1}
\begin{aligned}
\<\left(\I_N(g_N)\right)',Q(g_N)\>:=\&\frac{\d}{\d \tau}\I_N\left(g_N+\tau Q(g_N)\right)\bigg|_{\tau=0}\\=\&\frac{1}{N}\int_{\R^{3N}} Q(g_N)|\nabla\log g_N|^2\d V_N+\frac{2}{N}\int_{\R^{3N}}g_N\nabla \log g_N\cdot \nabla\frac{Q(g_N)}{g_N} \d V_N,
\end{aligned}
\end{equation}
where the gradient is taken with respect to $V_N\in\R^{3N}$, or alternatively, it has
\begin{equation}\label{def diff fisher2}
\begin{aligned}
\<\left(\I_N(g_N)\right)',Q(g_N)\>=\frac{2}{N}\int_{\R^{3N}}\frac{\nabla g_N}{g_N} \cdot \nabla Q(g_N)\d V_N-\frac{1}{N}\int_{\R^{3N}}\frac{|\nabla g_N|^2}{g_N} Q(g_N)\d V_N.
\end{aligned}
\end{equation}
We first remind the reader the results for the heat flow in order to introduce suitable notations and draw comparisons between the heat and the Landau operators \eqref{eq fN} acting on $g_N$. 

\begin{lemma}\label{lemma heat flow}
For any  positive smooth symmetric  function $g_N\in C^1\left([0,T],C^\infty(\R^{3N})\right)$ with rapid decay at infinity, the entropy and Fisher information of $g_N$ is decreasing along the heat flow, namely   
$$
\<\left(\H_N(g_N)\right)',\Delta g_N)\>\leq 0,
$$  
and
$$
\<\left(\I_N(g_N)\right)',\Delta g_N)\>\leq 0.
$$ 
\end{lemma}
\begin{proof}[Proof of Lemma \ref{lemma heat flow}]
By the definition and integration by parts, it is not hard to obtain the decay of the entropy, such as
$$
\begin{aligned}
 \<\left(\H_N(g_N)\right)',\Delta g_N)\>=\& \frac{1}{N}\sum_{i=1}^N\int_{\R^{3N}}\log g_N \big(\Delta_{v^i}g_N\big)\d V_N =  -\frac{1}{N}\sum_{i=1}^N\int_{\R^{3N}}\frac{ |\nabla_{v^i}g_N|^2}{g_N}\d V_N\leq 0.
\end{aligned}
$$
In terms of the Fisher information, we further define that for any $k=1,2,\ldots,N$,
\begin{equation}\label{def Ik}
\I_N(g_N)=\sum_{k=1}^N\I_N^k(g_N):=\frac{1}{N}\sum_{k=1}^N\int_{\R^{3N}}\frac{|\nabla_{v^k} g_N|^2}{g_N} \d V_N.
\end{equation}
Then it holds
$$
\begin{aligned}
\<\left(\I_N(g_N)\right)',\Delta 
g_N)\>\&=\Big\<\big(\sum_{k=1}^N\I_N^k(g_N)\big)',\sum_{i=1}^N\Delta_{v^i} g_N)\Big\>\\\&=\sum_{i=1}^N\Big\<\big(\I_N^i(g_N)\big)',\Delta_{v^i} g_N)\Big\>+\sum_{i\neq k}\Big\<\big(\I_N^k(g_N)\big)',\Delta_{v^i} g_N)\Big\>\\\&=\textbf{Term}_{i=k}+\textbf{Term}_{i\neq k},
\end{aligned}
$$
where by \eqref{def diff fisher1} and \eqref{def diff fisher2} we have
\begin{equation}\label{eq Term1}
\textbf{Term}_{i=k}=\frac{1}{N}\sum_{i=1}^N\bigg(\int_{\R^{3N}} \Delta_{v^i}g_N|\nabla_{v^i}\log g_N|^2\d V_N+2\int_{\R^{3N}}g_N\nabla_{v^i} \log g_N\cdot \nabla_{v^i}\frac{\Delta_{v^i}g_N}{g_N} \d V_N\bigg),
\end{equation}
and
\begin{equation}\label{eq Term2}
\textbf{Term}_{i\neq k}=\frac{2}{N}\sum_{i\neq k}\int_{\R^{3N}}\frac{\nabla_{v^k}g_N}{g_N}\cdot \nabla_{v^k}\big(\Delta_{v^i}g_N\big)\d V_N-\frac{1}{N}\sum_{i\neq k}\int_{\R^{3N}}\frac{|\nabla_{v^k}g_N|^2}{g_N^2}\big(\Delta_{v^i}g_N\big)\d V_N.
\end{equation}
The computation to simplify $\textbf{Term}_{i=k}$ can be found in, for example \cite{ledoux2022differentials}.
Applying Bochner's formula 
\begin{equation}\label{eq Bochner}
\nabla_{v^i}\log g_N \cdot \nabla_{v^i} \Delta_{v^i}\log g_N=\frac{1}{2}\Delta_{v^i} |\nabla_{v^i}\log  g_N|^2-|\nabla_{v^i}^2 \log g_N|^2,
\end{equation}
we have
\begin{equation}\label{ineq Term1}
\begin{aligned}
\textbf{Term}_{i=k}=\&\frac{1}{N}\sum_{i=1}^N\bigg(\int_{\R^{3N}} \Delta_{v^i}g_N|\nabla_{v^i}\log g_N|^2\d V_N+2\int_{\R^{3N}}g_N\nabla_{v^i} \log g_N\cdot \nabla_{v^i}\Delta_{v^i}\log g_N \d V_N\\\&+2\int_{\R^{3N}}g_N\nabla_{v^i} \log g_N\cdot \nabla_{v^i}|\nabla_{v^i}\log g_N |^2\d V_N\bigg)\\=\&\frac{1}{N}\sum_{i=1}^N\bigg(\int_{\R^{3N}} \Delta_{v^i}g_N|\nabla_{v^i}\log g_N|^2\d V_N+\int_{\R^{3N}}g_N\Delta_{v^i} | \nabla_{v^i}\log g_N|^2 \d V_N\\\&-2\int_{\R^{3N}}g_N | \nabla_{v^i}^2\log g_N|^2 \d V_N+2\int_{\R^{3N}}\nabla_{v^i}  g_N\cdot \nabla_{v^i}|\nabla_{v^i}\log g_N |^2\d V_N\bigg)\\=\&-\frac{2}{N}\sum_{i=1}^N\int_{\R^{3N}}g_N | \nabla_{v^i}^2\log g_N|^2 \d V_N\leq 0.
\end{aligned}
\end{equation}
For the first term in $\textbf{Term}_{i\neq k}$ \eqref{eq Term2}, we rewrite it into components and obtain the identity
\begin{equation}\label{eq term21}
\begin{aligned}
\int_{\R^{3N}}\frac{\nabla_{v^k}g_N}{g_N}\cdot \&\nabla_{v^k}\Big(\div_{v^i}\nabla_{v^i}g_N\Big)\d V_N= -\sum_{\alpha,\gamma=1}^3\int_{\R^{3N}}\p_{v^i_\alpha} \big(\frac{\p_{v^k_\gamma}g_N}{g_N}\big)  \p_{v^k_\gamma} \p_{v^i_\alpha}g_N\d V_N\\
=\& -\sum_{\alpha,\gamma}\int_{\R^{3N}} \frac{\p_{v^i_\alpha}\p_{v^k_\gamma}g_N}{g_N}  \p_{v^i_\alpha}\p_{v^k_\gamma} g_N\d V_N +\sum_{\alpha,\gamma}\int_{\R^{3N}} \frac{\p_{v^i_\alpha}g_N\p_{v^k_\gamma}g_N}{g_N^2}  \p_{v^i_\alpha}\p_{v^k_\gamma} g_N\d V_N.\\
=\& -\int_{\R^{3N}}\frac{\bar{M}^{ik}:\bar{M}^{ik}}{g_N}\d V_N+\int_{\R^{3N}}\frac{(\bar{\eta}^{i}\otimes\bar{\xi}^{k}):\bar{M}^{ik}}{g_N}\d V_N,
\end{aligned}
\end{equation}
by letting matrix $\bar{M}^{ik}$ be
$$
\bar{M}^{ik}:=\nabla_{v^i}\nabla_{v^k}g_N=\big(\p_{v^i_\alpha}\p_{v^k_\gamma} g_N\big)_{\alpha,\gamma=1,2,3},$$
and vectors $\bar{\eta}$, $\bar{\xi}$ be 
$$ 
\bar{\eta}^i:=\nabla_{v^i}g_N=\big(\p_{v^i_\alpha} g_N\big)_{\alpha=1,2,3},\quad \bar{\xi}^k:=\frac{\nabla_{v^k}g_N}{g_N}=\left(\frac{\p_{v^k_\gamma} g_N}{g_N}\right)_{\gamma=1,2,3}.
$$
The second term in $\textbf{Term}_{i\neq k}$ satisfies the following identity 
\begin{align}\label{eq term22}
\notag\int_{\R^{3N}}\frac{|\nabla_{v^k}g_N|^2}{g_N^2}\&\div_{v^i}\big(\nabla_{v^i}g_N\big)\d V_N=- \int_{\R^{3N}}\nabla_{v^i}\frac{|\nabla_{v^k}g_N|^2}{g_N^2}\cdot\big(\nabla_{v^i}g_N\big)\d V_N
\\\notag
=\&-2\sum_{\alpha,\gamma}\int_{\R^{3N}}\frac{\p_{v^k_\gamma}g_N}{g_N}\frac{\p_{v^i_\alpha}\p_{v^k_\gamma}g_N}{g_N}  \p_{v^i_\alpha}g_N\d V_N+2\sum_{\alpha,\gamma}\int_{\R^{3N}}\frac{(\p_{v^k_\gamma}g_N)^2\p_{v^i_\alpha}g_N}{g_N^3}\p_{v^i_\alpha}g_N\d V_N\\
=\&-2\int_{\R^{3N}}\frac{(\bar{\eta}^i\otimes\bar{\xi}^k):\bar{M}^{ik}}{g_N}\d V_N+2\int_{\R^{3N}}\frac{|\bar{\eta}^i|^2|\bar{\xi}^k|^2}{g_N}\d V_N.
\end{align}
Combining \eqref{eq term21} and \eqref{eq term22} yields
\begin{align}\label{ineq term2}
\notag\textbf{Term}_{i\neq k}=\&-\frac{2}{N}\sum_{i\neq k}\bigg(\int_{\R^{3N}}\frac{\bar{M}^{ik}:\bar{M}^{ik}}{g_N}\d V_N-2\int_{\R^{3N}}\frac{(\bar{\eta}^{i}\otimes\bar{\xi}^{k}):\bar{M}^{ik}}{g_N}\d V_N+\int_{\R^{3N}}\frac{|\bar{\eta}^{i}|^2|\bar{\xi}^{k}|^2}{g_N}\d V_N\bigg)\\  =\&-\frac{2}{N}\sum_{i\neq k}\bigg(\int_{\R^{3N}}\frac{|\bar{M}^{ik}-\bar{\eta}^{i}\otimes\bar{\xi}^{k}|^2}{g_N}\d V_N\bigg)\leq 0.
\end{align}
The second statement in the lemma holds by \eqref{ineq Term1} and \eqref{ineq term2} such as
$$
\begin{aligned}
\<\left(\I_N(g_N)\right)',\Delta g_N)\>=\textbf{Term}_{i=k}+\textbf{Term}_{i\neq k}\leq 0,
\end{aligned}
$$
which implies the desired decay of the Fisher information.
\end{proof}

Following the same strategy for the heat operator, we now prove the key ingredient in our work.

\begin{lemma}\label{lemma Liouviell flow} 
Assume that the potential $a(z)=a(|z|)$ satisfies 
\begin{equation}\label{ineq bound of a}
\frac{|z||a'(|z|)|}{a(|z|)}\leq\sqrt{22}.
\end{equation}
For any positive smooth symmetric function $g_N\in C^1\left([0,T],C^\infty(\R^{3N})\right)$ with rapid decay at infinity, the entropy and Fisher information of $g_N$ is decreasing along the flow given by the operator \begin{equation}\label{eq operator Q}Q(g_N)=\frac{1}{2N}\sum_{i\neq j}\div_{v^i-v^j}\left[a(v^i-v^j)\Pi(v^i-v^j)\nabla_{v^i-v^j}g_{N}\right],\end{equation}  namely   
$$
\<\left(\H_N(g_N)\right)',Q(g_N)\>\leq 0,    
$$
and
$$
\<\left(\I_N(g_N)\right)',Q(g_N)\>\leq 0.
$$ 
\end{lemma}
To prove the lemma above, we will take advantage of the following spectacular estimate given by \cite{guillen2025landau}. 
\begin{lemma}[{\cite[Proposition 10.2]{guillen2025landau}}]\label{lemma deacy of Fisher I2}
For any smooth positive function $G:\R^6\to(0,\infty)$ with rapid decay at infinity satisfying $G(v,w)=G(w,v)$, if the potential $a(z)=a(|z|)$ satisfies \eqref{ineq bound of a}, the Fisher information $\I_2(G)$ is decreasing along the flow given by  $\div_{v-w}\big[a(v-w)\Pi(v-w)\nabla_{v-w}G\big]$, namely,
$$
  \<(\I_2(G))',\div_{v-w}\big[a(v-w)\Pi(v-w)\nabla_{v-w}G\big]\>\leq 0. 
$$
\end{lemma}
Our Lemma \ref{lemma Liouviell flow} is valid for all power-law potentials $a(|z|)=|z|^{\gamma+2}$ in the range $\gamma\in[-3,1]$, covering the most physically relevant Coulomb case $\gamma=-3$. We also point out that the original constant in \eqref{ineq bound of a} was $\sqrt{19}$ in \cite{guillen2025landau}, which has been improved to $\sqrt{22}$ in \cite{ji2024bounds}. 
\begin{proof}[Proof of Lemma \ref{lemma Liouviell flow}] 
We evaluate the derivative of the entropy along the flow given by the operator \eqref{eq operator Q}
$$
\begin{aligned}
\Big\<\left(\H_N(g_N)\right)'\&,\frac{1}{2N}\sum_{i\neq j}\div_{v^i-v^j}\left[a(v^i-v^j)\Pi(v^i-v^j)\nabla_{v^i-v^j}g_{N}\right]\Big\>\\
=\&-\frac{1}{2N^2}\sum_{i\neq j}\int_{\R^{3N}}g_Na(v^i-v^j)\Pi(v^i-v^j):\nabla_{v^i-v^j}\log g_N\otimes  \nabla_{v^i-v^j}\log g_N\d V_N\leq 0,
\end{aligned}
$$
which is because the projection matrix $\Pi$ positive semi-definite.

In terms of the Fisher information, it holds 
$$
\begin{aligned}
\Big\<\Big(\I_N(g_N)\Big)'\&,\frac{1}{2N}\sum_{i\neq j}\div_{v^i-v^j}\left[a(v^i-v^j)\Pi(v^i-v^j)\nabla_{v^i-v^j}g_{N}\right]\Big\>\\=\&\frac{1}{2N}\sum_{i\neq j}\Big\<\Big(\I_N^i(g_N)\Big)'+\Big(\I_N^j(g_N)\Big)',\div_{v^i-v^j}\big[a(v^i-v^j)\Pi(v^i-v^j)\nabla_{v^i-v^j}g_N\big]\Big\>\\\&+\frac{1}{2N}\sum_{k\neq i\neq j}\Big\<\left(\I_N^k(g_N)\right)',\div_{v^i-v^j}\big[a(v^i-v^j)\Pi(v^i-v^j)\nabla_{v^i-v^j}g_N\big]\Big\>\\=\&\textbf{Term}_{k=i,j}+\textbf{Term}_{k\neq i\neq j},
\end{aligned}
$$
where we used the notation \eqref{def Ik}, and the convention $k\neq i\neq j$ denotes that $i,j,k$ are mutually different.

We will analyse $\textbf{Term}_{k=i,j}$ and $\textbf{Term}_{k\neq i\neq j}$ separately. We rewrite the first term into the derivative with respect to $(v^i,v^j)\in\R^6$,
\begin{equation}\label{eq Term k=ij}
\textbf{Term}_{k=i,j}=\frac{1}{2N^2}\sum_{i\neq j}\Big\<\big(\int_{\R^{3N}}\frac{|\nabla_{(v^i,v^j)} g_N|^2}{g_N} \d V_N\big)',\div_{v^i-v^j}\big[a(v^i-v^j)\Pi(v^i-v^j)\nabla_{v^i-v^j}g_N\big]\Big\>.
\end{equation}
For any fixed other variables $\hat{V}^{i,j}_N:=(v^1,\ldots,v^{i-1},v^{i+1},\ldots,v^{j-1},v^{j+1},\ldots,v^N)$, $g_N(\cdot,\cdot,\hat{V}^{i,j}_N):\R^6\to(0,\infty)$ is a smooth positive symmetric function with rapid decay at infinity. We can apply Lemma \ref{lemma deacy of Fisher I2} directly by letting variables $(v,w)=(v^i,v^j)$. For any $i\neq j$, the Fisher information $\I_2$ is decreasing along the direction $\div_{v^i-v^j}\big[a(v^i-v^j)\Pi(v^i-v^j)\nabla_{v^i-v^j}g_N\big]$, which holds pointwisely for every $\hat{V}^{i,j}_N\in \R^{3N-6}$. That is to say,  it holds for all $\hat{V}^{i,j}_N\in\R^{3N-6}$ that,
$$
\begin{aligned}
\big\<\&\left(\I_2\big(g_N(v^i,v^j,\hat{V}^{i,j}_N)\big)\right)',\div_{v^i-v^j}\big[a(v^i-v^j)\Pi(v^i-v^j)\nabla_{v^i-v^j}g_N(v^i,v^j,\hat{V}^{i,j}_N)\big]\big\>\\=\&\Big\<\Big(\frac{1}{2}\int_{\R^{6}}\frac{|\nabla_{(v^i,v^j)} g_N|^2}{g_N} \d v^i\d v^j\Big)',\div_{v^i-v^j}\big[a(v^i-v^j)\Pi(v^i-v^j)\nabla_{v^i-v^j}g_N\big]\Big\>(\hat{V}^{i,j}_N)\leq 0,
\end{aligned}
$$
which implies that the nonpositivity is also valid after integrating over $\hat{V}^{i,j}_N$ such as
$$
\textbf{Term}_{k=i,j}=\frac{1}{N^2}\sum_{i\neq j}\int_{\R^{3N-6}}\big\<\big(\I_2(g_N)\big)',\div_{v^i-v^j}\big[a(v^i-v^j)\Pi(v^i-v^j)\nabla_{v^i-v^j}g_N\big]\big\>\d \hat{V}^{i,j}_N\leq 0.
$$

\medskip

Now we focus on the index set $i\neq j\neq k$ where they are mutually different. The proof is somehow in the same spirit of the proof of Lemma \ref{lemma heat flow}. Let matrix $M^{ijk}$ be
\begin{equation}\label{def martix}
M^{ijk}:=\nabla_{v^i-v^j}\nabla_{v^k}g_N=\big(\p_{v^i_\alpha-v^j_\alpha}\p_{v^k_\gamma} g_N\big)_{\alpha,\gamma=1,2,3},
\end{equation} 
and vectors $\eta^{ij}$, $\xi^{k}$ be 
\begin{equation}\label{def vector} \eta^{ij}:=\nabla_{v^i-v^j}g_N=\big(\p_{v^i_\alpha-v^j_\alpha} g_N\big)_{\alpha=1,2,3},\quad \xi^{k}:=\frac{\nabla_{v^k}g_N}{g_N}=\left(\frac{\p_{v^k_\gamma} g_N}{g_N}\right)_{\gamma=1,2,3},
\end{equation} 
where $\p_{v^i_\alpha-v^j_\alpha} g_N=\p_{v^i_\alpha}g_N-\p_{v^j_\alpha}g_N$.
Note that $M^{ijk}$ and $\Pi$ are both symmetric matrices, i.e, $M^{ijk}_{\alpha\beta}=M^{ijk}_{\beta\alpha}$  and $\Pi_{\alpha\beta}=\Pi_{\beta\alpha}$. 

We write $\textbf{Term}_{k\neq i\neq j}$
explicitly by \eqref{def diff fisher2} such as
\begin{align}\label{eq Term kij}
\textbf{Term}_{k\neq i\neq j}=\&\frac{1}{2N^2}\sum_{i\neq j\neq k}\Big\<\left(\int_{\R^{3N}}\frac{|\nabla_{v^k} g_N|^2}{g_N} \d V_N\right)',\div_{v^i-v^j}\big[a(v^i-v^j)\Pi(v^i-v^j)\nabla_{v^i-v^j}g_N\big]\Big\>\notag\\=\&\frac{1}{N^2}\sum_{i\neq j\neq k}\int_{\R^{3N}}\frac{\nabla_{v^k}g_N}{g_N}\cdot \nabla_{v^k}\Big(\div_{v^i-v^j}\big[a(v^i-v^j)\Pi(v^i-v^j)\nabla_{v^i-v^j}g_N\big]\Big)\d V_N\notag\\
\&-\frac{1}{2N^2}\sum_{i\neq j\neq k}\int_{\R^{3N}}\frac{|\nabla_{v^k}g_N|^2}{g_N^2}\div_{v^i-v^j}\big[a(v^i-v^j)\Pi(v^i-v^j)\nabla_{v^i-v^j}g_N\big]\d V_N.
\end{align}
Then, the first term on the right-hand side of \eqref{eq Term kij} has the following identity by writing into component-wise form, for $\alpha,\beta,\gamma=1,2,3$, 
$$
\begin{aligned}
 \int_{\R^{3N}}\frac{\nabla_{v^k}g_N}{g_N}\&\cdot \nabla_{v^k}\Big(\div_{v^i-v^j}\big[a(v^i-v^j)\Pi(v^i-v^j)\nabla_{v^i-v^j}g_N\big]\Big)\d V_N\\
=\& \sum_{\gamma=1}^{3}\int_{\R^{3N}}\left(\frac{\p_{v^k_\gamma}g_N}{g_N}\right)\p_{v^k_\gamma}\Big(\sum_{\alpha=1}^{3}\p_{v^i_\alpha-v^j_\alpha}\big(\sum_{\beta=1}^{3} a(v^i-v^j)\Pi_{\alpha\beta}(v^i-v^j) \p_{v^i_\beta-v^j_\beta}g_N\big)\Big)\d V_N\\=\&-\sum_{\alpha,\beta,\gamma}\int_{\R^{3N}}\p_{v^i_\alpha-v^j_\alpha}\left(\frac{\p_{v^k_\gamma}g_N}{g_N}\right)\p_{v^k_\gamma}\Big( a(v^i-v^j)\Pi_{\alpha\beta}(v^i-v^j) \p_{v^i_\beta-v^j_\beta}g_N\Big)\d V_N,
\end{aligned}
$$
where the last step thanks to swapping two derivatives $\p_{v^k_\gamma}$ and $\p_{v^i_\alpha-v^j_\alpha}$ and then integration by parts. Notice that the cancellation holds 
$$
\p_{v^k_\gamma} \Big( a(v^i-v^j)\Pi_{\alpha\beta}(v^i-v^j)\Big)=0,
$$ 
when the index are mutually different $i\neq j\neq k$. Recall the notations \eqref{def martix} and \eqref{def vector}, we have the identity
\begin{equation}\label{eq J21}
\begin{aligned}
 \int_{\R^{3N}}\frac{\nabla_{v^k}g_N}{g_N}\&\cdot \nabla_{v^k}\Big(\div_{v^i-v^j}\big[a(v^i-v^j)\Pi(v^i-v^j)\nabla_{v^i-v^j}g_N\big]\Big)\d V_N\\
=\& -\sum_{\alpha,\beta,\gamma}\int_{\R^{3N}}\p_{v^i_\alpha-v^j_\alpha} \left(\frac{\p_{v^k_\gamma}g_N}{g_N}\right)  a(v^i-v^j)\Pi_{\alpha\beta}(v^i-v^j)\,\p_{v^k_\gamma} \p_{v^i_\beta-v^j_\beta}g_N\d V_N\\
=\& -\sum_{\alpha,\beta,\gamma}\int_{\R^{3N}} a(v^i-v^j)\frac{\p_{v^i_\alpha-v^j_\alpha}\p_{v^k_\gamma}g_N}{g_N}  \Pi_{\alpha\beta}(v^i-v^j)\,\p_{v^k_\gamma} \p_{v^i_\beta-v^j_\beta}g_N\d V_N\\\& +\sum_{\alpha,\beta,\gamma}\int_{\R^{3N}}a(v^i-v^j) \frac{\p_{v^i_\alpha-v^j_\alpha}g_N\p_{v^k_\gamma}g_N}{g_N^2}  \Pi_{\alpha\beta}(v^i-v^j)\,\p_{v^k_\gamma} \p_{v^i_\beta-v^j_\beta}g_N\d V_N\\=\&-\int_{\R^{3N}}a(v^i-v^j)\frac{M^{ijk}:\Pi M^{ijk}}{g_N}\d V_N+\int_{\R^{3N}}a(v^i-v^j)\frac{(\eta^{ij}\otimes\xi^{k}):\Pi M^{ijk}}{g_N}\d V_N.
\end{aligned}
\end{equation}
The second term on the right-hand side of \eqref{eq Term kij} satisfies the following identity with $i\neq j\neq k$ that 
\begin{align}\label{eq J22}
\notag\frac{1}{2}\int_{\R^{3N}}\frac{|\nabla_{v^k}g_N|^2}{g_N^2}\&\div_{v^i-v^j}\big[a(v^i-v^j)\Pi(v^i-v^j)\nabla_{v^i-v^j}g_N\big]\d V_N\\\notag= \&-\frac{1}{2}\int_{\R^{3N}}\nabla_{v^i-v^j}\frac{|\nabla_{v^k}g_N|^2}{g_N^2}\cdot\big[a(v^i-v^j)\Pi(v^i-v^j)\nabla_{v^i-v^j}g_N\big]\d V_N
\\=\&-\sum_{\alpha,\beta,\gamma}\int_{\R^{3N}}a(v^i-v^j)\frac{\p_{v^k_\gamma}g_N}{g_N}\frac{\p_{v^i_\alpha-v^j_\alpha}\p_{v^k_\gamma}g_N}{g_N}  \Pi_{\alpha\beta}(v^i-v^j)\p_{v^i_\beta-v^j_\beta}g_N\d V_N\\\notag\&+\sum_{\alpha,\beta,\gamma}\int_{\R^{3N}}a(v^i-v^j)\frac{(\p_{v^k_\gamma}g_N)^2\p_{v^i_\alpha-v^j_\alpha}g_N}{g_N^3}\Pi_{\alpha\beta}(v^i-v^j)\p_{v^i_\beta-v^j_\beta}g_N\d V_N\\\notag=\&-\int_{\R^{3N}}a(v^i-v^j)\frac{(\eta^{ij}\otimes\xi^{k}):\Pi M^{ijk}}{g_N}\d V_N+\int_{\R^{3N}}a(v^i-v^j)\frac{|\xi^k|^2\eta^{ij}\cdot(\Pi \eta^{ij})}{g_N}\d V_N.
\end{align}
Combining \eqref{eq J21} and \eqref{eq J22}, we observe that
$$
\begin{aligned}
\textbf{Term}_{k\neq i\neq j}=\&-\frac{1}{N^2}\sum_{i\neq j\neq k}\bigg(\int_{\R^{3N}}a(v^i-v^j)\frac{M^{ijk}:\Pi M^{ijk}}{g_N}\d V_N\\\&\qquad-2\int_{\R^{3N}}a(v^i-v^j)\frac{(\eta^{ij}\otimes\xi^{k}):\Pi M^{ijk}}{g_N}\d V_N+\int_{\R^{3N}}a(v^i-v^j)\frac{|\xi^{k}|^2\eta^{ij}\cdot\Pi \eta^{ij}}{g_N}\d V_N\bigg).
\end{aligned}
$$
Now we ignore the superscripts for a while, for the idempotent and symmetric projection matrix $\Pi$ and the symmetric matrix $M$, it holds
$$
\begin{aligned}
M:\Pi M=\&\sum_{\beta,\gamma,\delta}M_{\gamma\beta}\Pi_{\gamma\delta}M_{\beta\delta}=\sum_{\beta,\gamma,\delta}M_{\beta\gamma}\left(\sum_\alpha\Pi_{\alpha\gamma}\Pi_{\alpha\delta}\right)M_{\beta\delta}\\=\&\sum_{\alpha,\beta}\left(\sum_\gamma\Pi_{\alpha\gamma}M_{\beta\gamma}\right)\left(\sum_\delta\Pi_{\alpha\delta}M_{\beta\delta}\right)=\sum_{\alpha,\beta,\gamma,\delta}\Pi_{\alpha\gamma}M_{\beta\gamma}\Pi_{\alpha\delta}M_{\beta\delta}=\Pi M:\Pi M,
\end{aligned}
$$
and for any vectors $\eta$ and $\xi$ in $\R^3$, it holds
$$
\begin{aligned}
(\eta\otimes \xi):\Pi M=\&\sum_{\alpha,\beta,\gamma}\eta_\alpha \xi_\gamma \Pi_{\alpha\beta}M_{\gamma\beta}=\sum_{\alpha,\beta,\gamma}\eta_\alpha \xi_\gamma \left(\sum_\delta\Pi_{\alpha\delta}\Pi_{\beta\delta}\right)M_{\gamma\beta}\\=\&\sum_{\beta,\gamma,\delta}\left(\sum_\alpha\eta_\alpha \Pi_{\alpha\delta}\right) \xi_\gamma \left(\sum_\beta\Pi_{\delta\beta}M_{\gamma\beta}\right)=(\Pi \eta\otimes \xi):\Pi M,
\end{aligned}
$$
and 
$$
|\xi|^2\eta\cdot(\Pi \eta)=|\xi|^2(\Pi \eta)\cdot (\Pi \eta)=|\xi|^2|\Pi \eta|^2=(\Pi \eta\otimes \xi):(\Pi \eta\otimes \xi).
$$
Essentially, $\textbf{Term}_{k\neq i\neq j}$ can be simplified as
$$
\begin{aligned}
\textbf{Term}_{k\neq i\neq j}=\&-\frac{1}{N^2}\sum_{i\neq j\neq k}\int_{\R^{3N}}\frac{a}{g_N}\left(M:\Pi M-2(\eta\otimes \xi ):\Pi M+|\xi|^2\eta\cdot\Pi \eta\right)\d V_N \\ =\&-\frac{1}{N^2}\sum_{i\neq j\neq k}\int_{\R^{3N}}\frac{a}{g_N}\left(\Pi M:\Pi M-2(\Pi\eta\otimes \xi ):\Pi M+(\Pi \eta\otimes \xi):(\Pi \eta\otimes \xi)\right)\d V_N\\ =\&-\frac{1}{N^2}\sum_{i\neq j\neq k}\int_{\R^{3N}}\frac{a(v^i-v^j)}{g_N}\left|\Pi M^{ijk}-(\Pi \eta^{ij}\otimes \xi^{k})\right|^2\d V_N\leq 0, 
\end{aligned}
$$
where we used the Frobenius norm in the last line. Therefore, we get
$$
\begin{aligned}
\Big\<\left(\I_N(g_N)\right)',\frac{1}{2N}\sum_{i\neq j}\div_{v^i-v^j}\left[a(v^i-v^j)\Pi(v^i-v^j)\nabla_{v^i-v^j}g_{N}\right]\Big\>=\textbf{Term}_{k=i,j}+\textbf{Term}_{k\neq i\neq j}\leq 0,
\end{aligned}
$$
and conclude the second statement of the lemma.
\end{proof}

To rigorously prove Proposition \ref{prop well-posedness of fN} for solutions of the Liouville equation \eqref{eq fN}, we will apply a smooth cut-off near the origin and adding a vanishing linear diffusion. We introduce $f^\eps_N$ which satisfies
\begin{equation}\label{eq fNeps}
\begin{cases}
\displaystyle\p_t f_N^\eps=\frac{1}{2N}\sum_{i\neq j}\div_{v^i-v^j}\left[a^\eps(v^i-v^j)\Pi(v^i-v^j)\nabla_{v^i-v^j}f^\eps_{N}\right]+\eps \sum_{i=1}^N \Delta_{v^i}f_N^\eps,  \\ \\
      f_N^\eps(0)=(f^{0,\eps})^{\otimes N},  
\end{cases}
\end{equation}
where $a^\eps(z)=\chi(|z|/\eps)/|z|$ with $\chi\in C^\infty\big([0,\infty)\big)$, $\chi(r)=1$ when $r\geq 2$,  $\chi(r)=0$ when $r\leq 1$, and $f^{0,\eps}$ is positive and smooth with finite mass, energy, entropy and Fisher information.

By the standard theory of the linear parabolic equation, there exists a unique classical solution $f_N^\eps\in C^1\left([0,T],C^\infty(\R^{3N})\right) $ of \eqref{eq fNeps} which well-behaves at infinity. To be more precise, if the initial data can be bounded from above and below by Gaussians such as for fixed $\eps>0$, $C_1\exp(-\beta|v|^2)\leq f^{0,\eps}(v)\leq C_2\exp(-\beta|v|^2)$, then the upper and lower Gaussian bounds can be propagated in finite time to $f^{\eps}_N(t,V_N)$; and both $|\nabla \log f_N^\eps|$, $|\nabla^2 \log f_N^\eps|$ grow much slower than $\exp(\beta|V_N|^2)$. So all the integrals above in the proof are well-defined, and integration by parts and manipulations in the proof of Lemma \ref{lemma deacy of Fisher I2} are justified.

Notice that we can crucially choose our cut-off function $\chi$ in \eqref{eq fNeps} to ensure $a^\eps(z)=\chi(|z|/\eps)/|z|$ satisfies the bound \eqref{ineq bound of a}. Also, using some initial data $f^0$ fulfils Assumption \ref{assumption}, one can construct $f^{0,\eps}$ by the function $\chi$ and some smooth mollifier $\eta^\eps\to\delta_0$ such as 
$$
f^{0,\eps}(z)=(f^0\ast \eta^\eps(z)+\eps)\left(1-\chi(\eps|z|)\right)+\chi(\eps|z|)\exp(-|z|^2),
$$
which satisfies Gaussian bounds for fixed $\eps$ and converges to $f^0$ as $\eps\to0$. And $f^{0,\eps}$ has uniform-in-$\eps$ mass, momentum, energy, entropy and Fisher information.

From Lemma \ref{lemma heat flow} and Lemma \ref{lemma Liouviell flow}, we deduce that the solution $f_N^\eps$ of the regularised Liouville equation \eqref{eq fNeps} satisfies
$$
\begin{aligned}
\&\Big\<\left(\H_N(f_N^\eps)\right)',\frac{1}{2N}\sum_{i\neq j}\div_{v^i-v^j}\left[a^\eps(v^i-v^j)\Pi(v^i-v^j)\nabla_{v^i-v^j}f^\eps_{N}\right]+\eps \sum_{i=1}^N \Delta_{v^i}f_N^\eps\Big\>\\
\leq\& \frac{1}{2N}\sum_{i\neq j}\Big\<\left(\H_N(f_N^\eps)\right)',\div_{v^i-v^j}\left[a^\eps(v^i-v^j)\Pi(v^i-v^j)\nabla_{v^i-v^j}f^\eps_{N}\right]\Big\>\\ \&+ \eps \sum_{i=1}^N \Big\<\left(\H_N(f_N^\eps)\right)',\Delta_{v^i}f_N^\eps\Big\>\leq  0,
\end{aligned}
$$
and
$$
\begin{aligned}
\&\Big\<\left(\I_N(f_N^\eps)\right)',\frac{1}{2N}\sum_{i\neq j}\div_{v^i-v^j}\left[a^\eps(v^i-v^j)\Pi(v^i-v^j)\nabla_{v^i-v^j}f^\eps_{N}\right]+\eps \sum_{i=1}^N \Delta_{v^i}f_N^\eps\Big\>\\
\leq\& \frac{1}{2N}\sum_{i\neq j}\Big\<\left(\I_N(f_N^\eps)\right)',\div_{v^i-v^j}\left[a^\eps(v^i-v^j)\Pi(v^i-v^j)\nabla_{v^i-v^j}f^\eps_{N}\right]\Big\>\\ \&+ \eps \sum_{i=1}^N \Big\<\left(\I_N(f_N^\eps)\right)',\Delta_{v^i}f_N^\eps\Big\>\leq  0,
\end{aligned}
$$
which implies that, for  $\eps>0$ and any $t_2\geq t_1\geq 0$, it holds 
\begin{equation}\label{ineq eps entropy}
\begin{aligned}
\&\int_{\R^{3N}}f_N^\eps(t_2)\log f_N^\eps(t_2)\d V_N\leq \int_{\R^{3N}}f_N^\eps(t_1)\log f_N^\eps(t_1)\d V_N\\\leq\& \int_{\R^{3N}}f_N^\eps(0)\log f_N^\eps(0)\d V_N= N\int_{\R^{3}}f^{0,\eps}\log f^{0,\eps}\d v^1\leq C_1N,      
\end{aligned}
\end{equation}
and
\begin{equation}\label{ineq eps fisher}
\begin{aligned}
\&
\int_{\R^{3N}}\frac{|\nabla f_N^\eps(t_2)|^2}{f_N^\eps(t_2)}\d V_N\leq \int_{\R^{3N}}\frac{|\nabla f_N^\eps(t_1)|^2}{f_N^\eps(t_1)}\d V_N\\\leq\& \int_{\R^{3N}}\frac{|\nabla f_N^\eps(0)|^2}{f_N^\eps(0)}\d V_N=  N\int_{\R^{3}}\frac{|\nabla_{v^1} f^{0,\eps}|^2}{ f^{0,\eps}}\d v^1\leq C_2N.
\end{aligned}
\end{equation}
Also, we can check that the evolution of the second moment satisfies 
$$
\frac{\d}{\d t}\int_{\R^{3N}}f_N^\eps|V_N|^2\d V_N=\eps\sum_{i=1}^N\int_{\R^{3N}}\Delta_{v^i}f_N^\eps|V_N|^2\d V_N=6N\eps,
$$
then
\begin{equation}\label{eq uniform 2nd}
\int_{\R^{3N}}f_N^\eps(t)|V_N|^2\d V_N=N\int_{\R^{3}}f^{0,\eps}(v^1)|v^1|^2\d v^1+6N\eps t.   
\end{equation}
For any $T>0$, since the entropy bound \eqref{ineq eps entropy} and bound of the second moment \eqref{eq uniform 2nd} are uniform-in-$\eps$, by Dunford-Pettis criterion and a
diagonal extraction argument, there exist a function $f_N\in L^\infty\left(0,T;L^1_2(\R^{3N})\right)$ where $L^1_2$ is denoted by the $L^1$ function with finite second moment, for a dense set $\{\tau_k\}_{k\in\N}\subset[0,T]$,  up to some subsequence $\{\eps_\ell\}_{\ell\in\N}$ and $\eps_\ell\to 0$, 
\begin{equation}\label{conv eps dense weakly L1}
f_N^{\eps_\ell}(t)\rightharpoonup f_N(t)\quad\text{weakly in}\quad L^1(\R^{3N})\quad \forall t\in\{\tau_k\}_{k\in\N}.
\end{equation} 
Recall \eqref{ineq eps fisher}, we get the uniform-in-$\eps$ estimate
$$
\int_{\R^{3N}}\left|\nabla_{3N}\sqrt{f_{N}^\eps}\right|^2\d V_N<CN,
$$
namely $\sqrt{f_N^\eps}\in H^1(\R^{3N})$. 
For any bounded set 
$\B\subset \R^{3N}$ and the Sobolev compact embedding 
$H^1\left(\B\right)\hookrightarrow L^p\left(\B\right)$ with
$ p=6N/(3N-2)$, we are able to strengthen \eqref{conv eps dense weakly L1} to 
\begin{equation}\label{conv eps dense strongly Lp}
f_N^{\eps_\ell}(t)\to f_N(t)\quad\text{strongly in}\quad L^{3N/(3N-2)}_{loc}(\R^{3N})\quad \forall t\in\{\tau_k\}_{k\in\N}.
\end{equation}  
For any compactly supported test function $\psi_{N}\in W^{2,\infty}(\R^{3N})$ with its supported set $\supp\vphi_{N}\subset (\B_R)^N$, where $\B_R\subset\R^3$ is a closed ball with the radius $R>0$, and for $0\leq t_1\leq t_2\leq T$, we have  \begin{equation}\label{ineq equicountinuity}
\begin{aligned}
\bigg|\int_{\R^{3N}}\psi_{N}\& f_N^{\eps}(t_2)\d V_{N}-\int_{\R^{3N}}\psi_{N} f_N^{\eps}(t_1)\d V_{N}\bigg|\\=\&\bigg|\frac{1}{N}\sum_{i\neq j}^N\int_{t_1}^{t_2}\int_{\R^{3N}}A(v^i-v^j):(\nabla_{v^iv^i}^2\psi_{N}-\nabla_{v^iv^j}^2\psi_{N})f_N^\eps\d V_{N}\d t\\\&+\frac{1}{N}\sum_{i\neq j}^N\int_{t_1}^{t_2}\int_{\R^{3N}}B(v^i-v^j)\cdot (\nabla_{v^i}\psi_{N}-\nabla_{v^j}\psi_{N}) f_N^\eps\d V_{N}\d t\bigg|\\\leq\& N\int_{t_1}^{t_2}\left|\int_{\R^{3N}}A(v^1-v^2): (\nabla_{v^1v^1}^2\psi_{N}-\nabla_{v^1v^2}^2\psi_{N}) f_N^\eps\d V_{N} \right|\d t\\\&+ N\int_{t_1}^{t_2}\left|\int_{\R^{3N}}B(v^1-v^2)\cdot (\nabla_{v^1}\psi_{N}-\nabla_{v^2}\psi_{N}) f_N^\eps\d V_{N} \right|\d t.
\end{aligned}
\end{equation}
The integrand of the first term on the right-hand side of \eqref{ineq equicountinuity} holds
\begin{equation}\label{ineq fNeps A}
\begin{aligned}
\bigg|\int_{\R^{3N}}\& A(v^1-v^2): (\nabla_{v^1v^1}^2\psi_{N}-\nabla_{v^1v^2}^2\psi_{N}) f_N^\eps\d V_{N} \bigg|\\\leq\&\left(\left\|\nabla_{v^1v^1}^2\psi_N\right\|_{L^\infty(\R^{3N})}^2+\left\|\nabla_{v^1v^2}^2\psi_N\right\|_{L^\infty(\R^{3N})}^2\right)\int_{\supp\psi_N}\frac{1}{|v^1-v^2|} f_{N}^\eps\d V_{N} 
\\  \leq\&\left(\left\|\nabla_{v^1v^1}^2\psi_N\right\|_{L^\infty(\R^{3N})}^2+\left\|\nabla_{v^1v^2}^2\psi_N\right\|_{L^\infty(\R^{3N})}^2\right)\int_{\B_R\times \B_R}\frac{f_{N,2}^\eps(v^1,v^2)}{|v^1-v^2|} \d v^1\d v^2,  
\end{aligned}
\end{equation}
where $f_{N,2}^\eps$ is the second marginal of $f_N^\eps$. We denote 
\begin{equation}\label{exchanging}
V^{i,j}_N:=(v^1,\ldots,v^{i-1},v^j,v^{i+1},\ldots,v^{j-1},v^i,v^{j+1},\ldots,v^N)
\end{equation}
by exchanging $v^i$ and $v^j$ for $i<j$; thanks to the symmetry, the integrand of the second term on the right-hand side of \eqref{ineq equicountinuity} involving $B$ has the similar bound such as
\begin{align}\label{ineq fNeps B}
\notag\bigg|\int_{\R^{3N}}\&B(v^1-v^2)\cdot (\nabla_{v^1}\psi_{N}-\nabla_{v^2}\psi_{N}) f_{N}^\eps\d V_{N} \bigg|\\  \notag=\&\left|\int_{\R^{3N}}B(v^1-v^2)\cdot\left(\nabla_{v^1}\psi_{N}(V_N)-\nabla_{v^1}\psi_{N}(V_N^{1,2})\right) f_{N}^\eps\d V_{N} \right|\\\notag\leq\&2\int_{\R^{3N}}\left(\frac{\left|\nabla_{v^1}\psi_{N}(v^1,v^2)-\nabla_{v^1}\psi_{N}(v^2,v^2)\right|}{|v^1-v^2|}+\frac{\left|\nabla_{v^1}\psi_{N}(v^2,v^2)-\nabla_{v^1}\psi_{N}(v^2,v^1)\right|}{|v^1-v^2|}\right) \frac{f_{N}^\eps}{|v^1-v^2|}\d V_{N} 
\\  \leq\&2\left(\left\|\nabla_{v^1v^1}^2\psi_N\right\|_{L^\infty}^2+\left\|\nabla_{v^1v^2}^2\psi_N\right\|_{L^\infty}^2\right)\int_{\B_R\times \B_R}\frac{f_{N,2}^\eps(v^1,v^2)}{|v^1-v^2|} \d v^1\d v^2.
\end{align}
To control the common last integral in \eqref{ineq fNeps A} and \eqref{ineq fNeps B} with bounded integrating domain $\B_R\times\B_R$, we follow the technique used in \cite[Lemma 3.3]{fournier2014propagation} and manipulate changing of variables defined such as the linear transformation $\rho$
\begin{equation}\label{change of variables}
v^1,v^2\in\R^3,\quad \rho(v^1,v^2)=\frac{1}{\sqrt{2}}(v^1-v^2,v^1+v^2):=(w^1,w^2),
\end{equation}
with $|\det\rho|=|\det\rho^{-1}|=1$. The function obtained by changing of variables is denoted by $\tilde{G}$, namely
$$
\tilde{G}(w^1,w^2)=G\circ\rho^{-1}(w^1,w^2)=G(v^1,v^2).
$$
It is not hard to see $\tilde{G}$ and $G$ have the same Fisher information $\I_2(\tilde{G})=\I_2(G)$. For some bounded ball $\B_R$, it holds
\begin{equation}\label{ineq first marginal}
\begin{aligned}
\int_{\B_R\times \B_R}\&\frac{f_{N,2}^\eps(v^1,v^2)}{|v^1-v^2|} \d v^1\d v^2\leq\frac{1}{\sqrt{2}} \int_{\B_{2R}\times \B_{2R}}\frac{\tilde{f}_{N,2}^\eps(w^1,w^2)}{|w^1|} \d w^1\d w^2\leq \frac{1}{\sqrt{2}}\int_{\B_{2R}}\frac{\tilde{f}_{N,1}^\eps(w^1)}{|w^1|} \d w^1\\
\leq\& \frac{1}{\sqrt{2}}\left(\int_{\B_{2R}}\frac{1}{|w^1|^{3/2}}\d w^1\right)^{\frac{2}{3}}\left\|\tilde{f}_{N,1}^\eps\right\|_{L^3(\B_{2R})}\leq C\left\|\tilde{f}_{N,1}^\eps\right\|_{L^3(\B_{2R})},
\end{aligned}
\end{equation}
where $\tilde{f}_{N,1}$ is the first marginal of $\tilde{f}_{N,2}$ and the last inequality holds by the integrability. We notice that the compact embedding $H^1(\B_{2R})\hookrightarrow L^6(\B_{2R})$ holds for any $\B_{2R}\subset\R^3$ such as
$$
\left\|\tilde{f}_{N,1}^\eps\right\|_{L^3(\B_{2R})}=\left\|\sqrt{\tilde{f}_{N,1}^\eps}\right\|_{L^6(\B_{2R})}^2\leq C\left\|\nabla\sqrt{\tilde{f}_{N,1}^\eps}\right\|_{L^2(\B_{2R})}^2\leq \frac{C}{4}\I_1(\tilde{f}_{N,1}^\eps),
$$
where the estimate \eqref{ineq eps fisher} and Lemma \ref{lemma subadditivity} implies that for any $t\in[0,T]$
$$
\I_1(\tilde{f}_{N,1}^\eps(t))\leq \I_2(\tilde{f}_{N,2}^\eps(t))=\I_2(f_{N,2}^\eps(t))\leq \I_N(f_{N}^\eps(t))\leq \I_1(f^{0,\eps})<C.
$$
Therefore we get the estimate of  \eqref{ineq first marginal}, such as for any $\eps$
\begin{equation}\label{ineq double integral}
\sup_{t\in[0,T]}\int_{\B_R\times \B_R}\frac{f_{N,2}^\eps(t,v^1,v^2)}{|v^1-v^2|} \d v^1\d v^2\lesssim \sup_{t\in[0,T]}\I_N(f_N^\eps(t))\leq \I_1(f^{0,\eps})<C,
\end{equation}
where the constant $C$ does not dependent on $\eps$ and $t$.
Substituting the bound \eqref{ineq double integral} into \eqref{ineq fNeps A} and \eqref{ineq fNeps B}, we get the boundedness of the right-hand side of \eqref{ineq equicountinuity}
$$
\begin{aligned}
\int_{t_1}^{t_2}\left|\int_{\R^{3N}}A(v^1-v^2): (\nabla_{v^1v^1}^2\psi_{N}-\nabla_{v^1v^2}^2\psi_{N}) f_N^\eps\d V_{N} \right|\d s\leq C(\psi_N)|t_2-t_1|
,
\end{aligned}
$$
and 
$$
\begin{aligned}
\int_{t_1}^{t_2}\left|\int_{\R^{3N}}B(v^1-v^2)\cdot (\nabla_{v^1}^2\psi_{N}-\nabla_{v^2}^2\psi_{N}) f_N^\eps\d V_{N} \right|\d s\leq C(\psi_N)|t_2-t_1|.
\end{aligned}
$$
That is to say, for any fixed $N\in\N$, the sequence $\{\int\psi_{N} f_N^\eps\d V_{N}\}_{\eps>0}$ is equicontinuous on $[0,T]$ for any compactly supported $\psi_N\in W^{2,\infty}(\R^{3N})$, such as
\begin{equation*}\label{bound equicontinuity}
 \left|\int_{\R^{3N}}\psi_{N} f_N^\eps(t_2)\d V_{N}-\int_{\R^{3N}}\psi_{N} f_N^\eps(t_1)\d V_{N}\right|\leq C(\psi_N)N |t_2-t_1|,   
\end{equation*}
where the constant $C(\psi_N)$ does not dependent on $\eps$, $t_1$ or $t_2$. By the Arzelà–Ascoli theorem, when $\eps_\ell$ goes to $0$ in the subsequence $\{\eps_\ell\}_{\ell\in\N}$ extracted in \eqref{conv eps dense weakly L1}, we refer that it holds for any compactly supported $\psi_N\in W^{2,\infty}(\R^{3N})$
\begin{equation}\label{conv all in time}
\int_{\R^{3N}}\psi_{N} f_N^{\eps_\ell}(t)\d V_{N}\to \int_{\R^{3N}}\psi_{N} f_N(t)\d V_{N}, \quad \forall t\in[0,T],
\end{equation}
where the map $t\mapsto \int_{\R^{3N}}\psi_{N} f_N(t)\d V_{N}$ is Lipschitz  continuous. Combining the convergence \eqref{conv eps dense strongly Lp} on some dense time set and the continuity in time \eqref{conv all in time}, we summarise that the argument above actually gives the proof of the proposition below. 
\begin{proposition}\label{Decay of Fisher information}
For any fixed $N\in\N$ and $T>0$, there exits a $f_{N}\in L^\infty\left(0,T;L^1_2(\R^{3N})\right)$where $t\mapsto \int_{\R^{3N}}\psi_{N} f_N(t)\d V_{N}$ is Lipschitz continuous for any compactly supported $\psi_N\in W^{2,\infty}(\R^{3N})$.  And there exists some subsequence of the $\{f_{N}^{\eps_\ell}\}_{\ell\in\N}$  satisfying the regularised equation \eqref{eq fNeps}, which converges to  $f_N$ uniformly on some dense time set $\{\tau_k\}_{k\in\N}\subset[0,T]$ when $\eps_\ell\to0$, such as
\begin{equation}\label{convergence fNeps}
f_{N}^{\eps_\ell}(t) \to f_N(t)\quad\text{strongly in}\quad L^{1}_{loc}(\R^{3N}),\quad \forall t\in\{\tau_k\}_{k\in\N}.
\end{equation}
Moreover, the entropy and Fisher information of $f_N$ are monotonically decreasing in time, namely
\begin{equation}\label{decay of entropy fN}
\frac{1}{N}\int_{\R^{3N}}f_{N}(t_2)\log f_{N}(t_2)\d V_N\leq \frac{1}{N}\int_{\R^{3N}}f_{N}(t_1)\log f_{N}(t_1)\d V_N\leq \int_{\R^{3}}f^0\log f^0\d v^1,
\end{equation}
and
\begin{equation}\label{decay of fisher fN}
\frac{1}{N}\int_{\R^{3N}}\frac{|\nabla f_{N}(t_2)|^2}{f_{N}(t_2)}\d V_N\leq \frac{1}{N}\int_{\R^{3N}}\frac{|\nabla f_{N}(t_1)|^2}{f_{N}(t_1)}\d V_N\leq  \int_{\R^{3}}\frac{|\nabla_{v^1} f^0|^2}{ f^0}\d v^1.
\end{equation}
In particular, under Assumption \ref{assumption},  the (renormalised) entropy $\H_N(f_N)$ and (renormalised) Fisher information $\I_N(f_N)$ are uniform-in-$N$ bounded.
\end{proposition}

To show the existence part of Proposition \ref{prop well-posedness of fN}, it remains to pass to the limit. However, the argument closely parallels that in  Section \ref{sec pass limit}, where we provide the details of passing from the marginal $f_{N,m}$ to the hierarchy $f_m$. To avoid redundancy, we omit a similar discussion for the regularised Liouville equation.

The uniqueness part of Proposition \ref{prop well-posedness of fN} follows from the remark below. Also, the uniqueness implies that every convergent subsequence in Proposition \ref{Decay of Fisher information} has the same limiting point, leading to a complete clarification of the limits.
\begin{remark}[Stability estimate]\label{Liouville uniqueness}
We assume two weak solutions $f_N'$ and $ f_N''$ in the sense of Definition \ref{def fN} with the same initial data $f_N'(0)=f_N''(0)=(f^0)^{\otimes N}$. Boundedness of the entropy (Proposition \ref{Decay of Fisher information}) enables us to  evaluate the relative entropy between $f_N'$ and $f_N''$ 
$$
\begin{aligned}
\frac{\d}{\d t}\int_{\R^{3N}}f_N'\log\frac{f_N'}{f_N''}\d V_N=\&    \int_{\R^{3N}}\p_t f_N'\log\frac{f_N'}{f_N''}\d V_N-\int_{\R^{3N}}\frac{f_N'}{f_N''}\p_t f_N'' \d V_N\\
=\&-\frac{1}{2N}\sum_{i\neq j}\int_{\R^{3N}}A(v^i-v^j):\nabla_{v^i-v^j}f_N'\otimes\nabla_{v^i-v^j} \log\frac{f_N'}{f_N''}\d V_N\\
\&+\frac{1}{2N}\sum_{i\neq j}\int_{\R^{3N}}A(v^i-v^j):\nabla_{v^i-v^j}f_N''\otimes\nabla_{v^i-v^j} \frac{f_N'}{f_N''}\d V_N\\
=\&-\frac{1}{2N}\sum_{i\neq j}\int_{\R^{3N}}f_N'A(v^i-v^j):\nabla_{v^i-v^j}\log\frac{f_N'}{f_N''}\otimes\nabla_{v^i-v^j} \log\frac{f_N'}{f_N''}\d V_N\leq 0,
\end{aligned}
$$
which implies that the weak solution of \eqref{eq fN} is unique. 
\end{remark}

\section{Convergence to the Landau hierarchy}\label{sec convergence}

In the previous section, we showed there exists a unique weak solution of the Liouville equation \eqref{eq fN} satisfying the decay of entropy \eqref{fN entropy decay} and Fisher information \eqref{fN Fisher decay}. In this section, we begin the proof of our main result Theorem \ref{thm main}. Then, for any integer $m$ between $1$ and $N-1$, we integrate over variables from $v^{m+1}$ to $v^N$ of $f_N$ to get the unique $m$-marginal $f_{N,m}$.
By Remark \ref{lemma subadditivity}, we refer that the marginals satisfies the following uniform-in-$N$ bounds: 
\begin{equation}\label{ineq fNm entropy}
\begin{aligned}
\&\sup_{t\in[0,T]}\int_{\R^{3m}}f_{N,m}(t)\log f_{N,m}(t) \d V_{m}\leq \frac{m}{N}\sup_{t\in[0,T]}\int_{\R^{3N}}f_N(t)\log f_N(t)\d V_N\\\leq\& m\int_{\R^3}f^0\log f^0\d v^1<+\infty,
\end{aligned} 
\end{equation}
and
\begin{equation}\label{ineq fNm fisher}
\begin{aligned}
\&\sup_{t\in[0,T]}\int_{\R^{3m}}f_{N,m}(t)|\nabla_{3m}\log f_{N,m}(t)|^2 \d V_{m}\leq \frac{m}{N}\sup_{t\in[0,T]}\int_{\R^{3N}}f_N(t)|\nabla\log f_N(t)|^2 \d V_N\\\leq\& m\int_{\R^3}f^0|\nabla \log f^0|^2\d v^1<+\infty,
\end{aligned} 
\end{equation}
which provide us sufficient compactness to prove large $N$ limit. And the marginals $f_{N,m}$ satisfy the BBGKY hierarchy \eqref{BBGKY} in weak sense defined below. 
\begin{definition}[Weak solutions of the BBGKY hierarchy]\label{def BBGKY hierarchy}
For any time $T>0$, a weak solution of the BBGKY hierarchy \eqref{BBGKY} with $f_{N,m}(0)=(f^0)^{\otimes m}$ satisfies the following
weak form of hierarchy, for any compactly supported $\vphi_{m}\in W^{4,\infty}(\R^{3m})$
\begin{equation}\label{eq weak BBGKY}
\begin{aligned}
\int_{\R^{3m}}\vphi_{m}\& f_{N,m}(T)\d V_{m}-\int_{\R^{3m}}\vphi_{m} (f^0)^{\otimes m}\d V_{m}\\=\&\frac{1}{N}\sum_{i\neq j}\int_0^T\int_{\R^{3m}}A(v^i-v^j):(\nabla_{v^iv^i}^2\vphi_{m}-\nabla_{v^iv^j}^2\vphi_{m})f_{N,m}\d V_{m}\d t\\\&+\frac{1}{N}\sum_{i\neq j}\int_0^T\int_{\R^{3m}}B(v^i-v^j)\cdot (\nabla_{v^i}\vphi_{m}-\nabla_{v^j}\vphi_{m}) f_{N,m}\d V_{m}\d t\\
\&+ \frac{N-m}{N}\sum_{i=1}^m \int_0^T\int_{\R^{3(m+1)}}  A(v^i-v^{m+1}):\nabla_{v^iv^i}^2\vphi_{m}\,f_{N,m+1}\d V_{m+1}\d t\\
\&+ \frac{2(N-m)}{N}\sum_{i=1}^m \int_0^T\int_{\R^{3(m+1)}}  B(v^i-v^{m+1})\cdot \nabla_{v^{i}}\vphi_{m}\,f_{N,m+1}\,\d V_{m+1}\d t.
\end{aligned}
\end{equation}
Additionally, the entropy and Fisher information of solutions are uniformly bounded as stated in \eqref{ineq fNm entropy} and \eqref{ineq fNm fisher}.
\end{definition}
The weak form \eqref{eq weak BBGKY} can be interpreted as choosing the test function $\vphi_N = \vphi_m \otimes (1^{\otimes (N-m)})$ in \eqref{eq weak Liouville}. Similarly to \eqref{exchanging}, we denote $V^{i,m+1}_m$ as
$$
V^{i,m+1}_m:=(v^1,\ldots,v^{i-1},v^{m+1},v^{i+1},\ldots,v^m).
$$
Swapping $v^i$ and $v^{m+1}$ in the last integral of \eqref{eq weak BBGKY}, and using the symmetry of $f_{N,m+1}$ along with the anti-symmetry of vector field $B$, we get the identity
$$
\begin{aligned}
\&\int_{\R^{3(m+1)}}  B(v^i-v^{m+1})\cdot \nabla_{v^{i}}\vphi_{m}(V_m)\,f_{N,m+1}\,\d V_{m+1}\\=\&-\int_{\R^{3(m+1)}}  B(v^i-v^{m+1})\cdot \nabla_{v^{i}}\vphi_{m}(V^{i,m+1}_m)\,f_{N,m+1}\,\d V_{m+1}.
\end{aligned}
$$
Moreover, the last integral in \eqref{eq weak BBGKY} can be written as
\begin{equation}\label{eq Vmi}
\begin{aligned}
2\int_{\R^{3(m+1)}}\&B(v^1-v^{m+1})\cdot \nabla_{v^{1}}\vphi_{m}(V_m)f_{N,m+1}\d V_{m+1}\\
=\&\int_{\R^{3(m+1)}}B(v^1-v^{m+1})\cdot \left(\nabla_{v^{1}}\vphi_{m}(V_m)-\nabla_{v^{1}}\vphi_{m}(V_m^{i,m+1})\right)f_{N,m+1}\d V_{m+1}.
\end{aligned}
\end{equation}
We now aim to find the large-$N$ limit of the weak solution of the BBGKY hierarchy. Again, the uniform-in-$N$ entropy bound \eqref{ineq fNm entropy} and uniform-in-$N$ bound of the second moment hold; by Dunford-Pettis criterion and a
diagonal extraction argument, for any $T>0$, there exist a function $f_m\in L^\infty\left(0,T;L^1_2(\R^{3m})\right)$, for a dense set $\{\tau_k\}_{k\in\N}\subset[0,T]$,  up to some subsequence $\{N_\ell\}_{\ell\in\N}$ and $N_\ell\to \infty$, such that
\begin{equation}\label{conv N dense weakly L1}
f_{N_\ell,m}(t)\rightharpoonup f_{m}(t)\quad\text{weakly in}\quad L^1(\R^{3m})\quad \forall t\in\{\tau_k\}_{k\in\N}\subset[0,T].
\end{equation} 
Recall \eqref{ineq fNm fisher}, we get the uniform-in-$N$ estimate
$$
\int_{\R^{3m}}\left|\nabla_{3m}\sqrt{f_{N,m}}\right|^2\d V_m<C(m),
$$
Furthermore, the compact embedding $ H^1(\mathcal{B}) \hookrightarrow L^p(\mathcal{B}) $ for any bounded set $\B\subset\R^{3m}$ and $p=6m/(3m-2)$ implies that the convergence in \eqref{conv N dense weakly L1} holds strongly in $L^{3m/(3m-2)}_{loc}$ as $N_\ell$ goes to $\infty$. That is,
\begin{equation}\label{convergence Lploc}
f_{N_\ell,m}(t)\to f_{m}(t)\quad\text{strongly in}\quad L^{3m/(3m-2)}_{loc}(\R^{3m})\quad \forall t\in\{\tau_k\}_{k\in\N}\subset[0,T].
\end{equation} 
For any compactly supported $\psi_{m}\in W^{2,\infty}(\R^{3m})$ with the supported set $\supp\vphi_{m}\subset (\B_R)^m$, we have
\begin{equation}\label{ineq equicountinuity N}
\begin{aligned}
\bigg|\int_{\R^{3m}}\psi_{m}\& f_{N,m}(t_2)\d V_{m}-\int_{\R^{3m}}\psi_{m} f_{N,m}(t_1)\d V_{m}\bigg|\\\leq\& \frac{m(m-1)}{N}\int_{t_1}^{t_2}\left|\int_{\R^{3m}}A(v^1-v^2): (\nabla_{v^1v^1}^2\psi_{m}-\nabla_{v^1v^2}^2\psi_{m}) f_{N,m}\d V_{m} \right|\d s\\\&+ \frac{m(m-1)}{N}\int_{t_1}^{t_2}\left|\int_{\R^{3m}}B(v^1-v^2)\cdot (\nabla_{v^1}^2\psi_{N}-\nabla_{v^2}^2\psi_{m}) f_{N,m}\d V_{m}  \right|\d s\\
\&+ \frac{m(N-m)}{N}\int_{t_1}^{t_2}\left|\int_{\R^{3(m+1)}}  A(v^1-v^{m+1}):\nabla_{v^1v^1}^2\psi_{m}\,f_{N,m+1}\d V_{m+1}\right|\d s\\
\&+ \frac{2m(N-m)}{N}\int_{t_1}^{t_2}\left|\int_{\R^{3(m+1)}}  B(v^1-v^{m+1})\cdot \nabla_{v^{1}}\psi_{m}\,f_{N,m+1}\,\d V_{m+1}\right|\d s.
\end{aligned}
\end{equation}
Similarly to \eqref{ineq fNeps A} and \eqref{ineq fNeps B}, the first two terms on the right-hand side of \eqref{ineq equicountinuity N} have the estimate
$$
\begin{aligned}
\&\bigg|\int_{\R^{3m}}A(v^1-v^2): (\nabla_{v^1v^1}^2\psi_{m}-\nabla_{v^1v^2}^2\psi_{m}) f_{N,m}\d V_{m} \bigg|
\\  \leq\&\left(\left\|\nabla_{v^1v^1}^2\psi_m\right\|_{L^\infty}+\left\|\nabla_{v^1v^2}^2\psi_m\right\|_{L^\infty}\right)\int_{\B_R\times \B_R}\frac{f_{N,2}(v^1,v^2)}{|v^1-v^2|} \d v^1\d v^2,   
\end{aligned}
$$
and
$$
\begin{aligned}
\&\bigg|\int_{\R^{3m}}B(v^1-v^2)\cdot (\nabla_{v^1}\psi_{m}-\nabla_{v^2}\psi_{m}) f_{N,m}\d V_{m} \bigg|\\\leq\&2\left(\left\|\nabla_{v^1v^1}^2\psi_m\right\|_{L^\infty}+\left\|\nabla_{v^1v^2}^2\psi_m\right\|_{L^\infty}\right)\int_{\B_R\times \B_R}\frac{f_{N,2}(v^1,v^2)}{|v^1-v^2|} \d v^1\d v^2.
\end{aligned}
$$
Notice that the integral over $v^{m+1}$ is in the whole space; the terms with $f_{N,m+1}$ can be bounded in the form
$$
\begin{aligned}
\bigg|\int_{\R^{3(m+1)}}\&A(v^1-v^{m+1}): \nabla_{v^1v^1}^2\psi_{m} f_{N,m+1}\d V_{m+1} \bigg|
\\  \leq\&\left\|\nabla_{v^1v^1}^2\psi_m\right\|_{L^\infty}\int_{\B_R\times \R^{3m}}\frac{f_{N,m+1}(V_{m+1})}{|v^1-v^{m+1}|} \d V_{m+1}\\\leq\&\left\|\nabla_{v^1v^1}^2\psi_m\right\|_{L^\infty}\left(\int_{\B_R\times \B_{2R}^c}\frac{f_{N,2}(v^1,v^2)}{|v^1-v^2|} \d v^1\d v^2+\int_{\B_{2R}\times \B_{2R}}\frac{f_{N,2}(v^1,v^2)}{|v^1-v^2|} \d v^1\d v^2\right)\\\leq\&\left\|\nabla_{v^1v^1}^2\psi_m\right\|_{L^\infty}\left(\frac{1}{R}+\int_{\B_{2R}\times \B_{2R}}\frac{f_{N,2}(v^1,v^2)}{|v^1-v^2|} \d v^1\d v^2\right),  
\end{aligned}
$$
and using the identity \eqref{eq Vmi},
$$
\begin{aligned}
\bigg|\int_{\R^{3(m+1)}}\&B(v^1-v^{m+1})\cdot \left(\nabla_{v^{1}}\psi_{m}(V_m)-\nabla_{v^{1}}\psi_{m}(V_m^{i,m+1})\right)f_{N,m+1}\d V_{m+1}\bigg|\\\leq\&2\int_{\R^{3(m+1)}}\frac{1}{|v^1-v^{m+1}|^2}\left|\nabla_{v^1}\psi_{m}(v^1)-\nabla_{v^1}\psi_{m}(v^{m+1})\right| f_{N,m+1}\d V_{m} 
\\   \leq\&2\left\|\nabla_{v^1v^1}^2\psi_m\right\|_{L^\infty}\int_{\B_R\times \R^{3m}}\frac{f_{N,m+1}(V_{m+1})}{|v^1-v^{m+1}|} \d V_{m+1}
\\  \leq\&2\left\|\nabla_{v^1v^1}^2\psi_m\right\|_{L^\infty}\left(\frac{1}{R}+\int_{\B_{2R}\times \B_{2R}}\frac{f_{N,2}(v^1,v^2)}{|v^1-v^2|} \d v^1\d v^2\right).
\end{aligned}
$$
By applying the same change of variables argument as in the previous section (see from \eqref{ineq first marginal} to \eqref{ineq double integral}), we have
\begin{equation}\label{ineq double integral N}
\sup_{t\in[0,T]}\int_{\B_R\times \B_R}\frac{f_{N,2}(v^1,v^2)}{|v^1-v^2|} \d v^1\d v^2\lesssim\, \sup_{t\in[0,T]}\I_N(f_N)\leq \I_1(f^{0})<C. 
\end{equation}
Therefore, by collecting the estimates above for the right-hand side of \eqref{ineq equicountinuity N}, we arrive at the following lemma, which will be useful in the next section.
\begin{lemma}\label{lemma uniform bounds with time}
For any $0\leq t_1\leq t_2\leq T$, for any compactly supported $\psi_{m}\in W^{2,\infty}(\R^{3m})$, the $m$-marginal and $(m+1)$-marginal of the weak solution of the Liouville equation  \eqref{eq fN} satisfy the following uniform-in-$N$ bounds:
\begin{equation}\label{ineq fNm A}
\begin{aligned}
\int_{t_1}^{t_2}\bigg|\int_{\R^{3N}}A(v^1-v^2):(\nabla_{v^1v^1}^2\psi_{m}-\nabla_{v^1v^2}^2\psi_{m})f_{N,m}\d V_{m}\bigg|\d s\leq C(\psi_m, f^0)|t_2-t_1|
,
\end{aligned}
\end{equation}
\begin{equation}\label{ineq fNm B}
\begin{aligned}
\int_{t_1}^{t_2}\bigg|\int_{\R^{3N}}B(v^1-v^2)\cdot (\nabla_{v^1}\psi_m-\nabla_{v^2}\psi_m) f_{N,m}\d V_{m}\bigg|\d s\leq C(\psi_m, f^0)|t_2-t_1|,
\end{aligned}
\end{equation}
\begin{equation}\label{ineq fNm+1 A}
\int_{t_1}^{t_2}\bigg|\int_{\R^{3(m+1)}}A(v^1-v^{m+1}): \nabla_{v^1v^1}^2\psi_{m} f_{N,m+1}\d V_{m+1}\bigg|\d s\leq C(\psi_m, f^0)|t_2-t_1|,
\end{equation}
and
\begin{equation}\label{ineq fNm+1 B}
\int_{t_1}^{t_2}\bigg|\int_{\R^{3(m+1)}}B(v^1-v^{m+1})\cdot \nabla_{v^{1}}\psi_{m}f_{N,m+1}\d V_{m+1}\d t\bigg|\d s\leq C(\psi_m, f^0)|t_2-t_1|.
\end{equation}
\end{lemma}
Plugging in bounds \eqref{ineq fNm A}-\eqref{ineq fNm+1 B} into \eqref{ineq equicountinuity N}, we get that, for any fixed $m\in\N$, the sequence \\
$\{\int\psi_m f_{N,m}\d V_{m}\}_{N\in\N}$ is equicontinuous on $[0,T]$, such as
\begin{equation}\label{bound equicontinuity N}
 \left|\int_{\R^{3m}}\psi_{m} f_{N,m}(t_2)\d V_{m}-\int_{\R^{3m}}\psi_{m} f_{N,m}(t_1)\d V_{m}\right|\leq C(m) |t_2-t_1|,   
\end{equation}
where the constant $C(m)$ does not dependent on $N$, $t_1$ or $t_2$. By the Arzelà–Ascoli theorem, when $N_\ell$ goes to $\infty$, we refer that it holds for any compactly supported $\psi_{m}\in W^{2,\infty}(\R^{3m})$ that
\begin{equation}\label{conv all in time N}
\int_{\R^{3m}}\psi_{m} f_{N_\ell,m}(t)\d V_{m}\to \int_{\R^{3m}}\psi_{m} f_m(t)\d V_{m}, \quad \forall t\in[0,T],
\end{equation}
where $t\mapsto \int_{\R^{3m}}\psi_{m} f_m(t)\d V_{m}$ is Lipschitz continuous. Combining the convergence \eqref{convergence Lploc} on the dense time set and \eqref{conv all in time N}, we obtain the following convergence result for BBGKY hierarchy. 
\begin{proposition}[$L^1$-convergence]\label{prop strong convergence}
For any fixed $m\in\N$ and $T>0$, there exists a function $f_m\in L^\infty\left(0,T;L^1_2(\R^{3m})\right)$, where $t\mapsto \int_{\R^{3m}}\psi_{m} f_m(t)\d V_{m}$ is Lipschitz continuous for any compactly supported $\psi_m\in W^{2,\infty}(\R^{3m})$. And there exits some subsequence of the $\{f_{N_\ell,m}\}_{\ell\in\N}$ satisfying the BBGKY hierarchy \eqref{BBGKY}, which converges to $f_m$ uniformly on some dense time set $\{\tau_k\}_{k\in\N}\subset[0,T]$ when $N_\ell\to\infty$, such as
$$
f_{N_\ell,m}(t)\to f_{m}(t)\quad\text{strongly in}\quad L^{3m/(3m-2)}_{loc}(\R^{3m})\quad \forall t\in\{\tau_k\}_{k\in\N}\subset[0,T].
$$
In particularly, the strong convergence holds in $L^1_{loc}$ for all fixed $m\in\N$ such as
\begin{equation}\label{convergence L1loc}
f_{N_\ell,m}(t)\to f_m(t) \quad\text{strongly in}\quad L^{1}_{loc}(\R^{3m})\quad \forall t\in\{\tau_k\}_{k\in\N}\subset[0,T].
\end{equation}
\end{proposition}

\section{Pass to the limit}\label{sec pass limit}
To complete the proof of Theorem \ref{thm main}, it remains to justify that the limit $f_m$ satisfies Definition \ref{def Landau hierarchy} for the Landau hierarchy. In other words, we will show the proposition below.
\begin{proposition}\label{prop pass to the limit} 
Any limiting point $f_m$ obtained in Proposition \ref{prop strong convergence}, is a weak solution of the Landau hierarchy \eqref{eq Landau hierarchy} in the sense of Definition \ref{def Landau hierarchy}.    
\end{proposition}
We shall give the detailed proof in this section.  Also, the argument can be easily adapted to complete the proof of existence of the Liouville equation \eqref{eq fN} mentioned in the end of Section \ref{section uniform estimate}.

Our aim is to pass to the limit from the weak form of the BBGKY hierarchy \eqref{eq weak BBGKY} to the  weak form of the Landau hierarchy \eqref{eq weak Landau hierarchy}. Firstly, for any $T>0$, assuming that $\vphi_m\in W^{4,\infty}(\R^{3m})$ with the compact supported set $supp\,\vphi_m\subset (\B_R)^m$, it holds
$$
\left|\int_{\R^{3m}}\vphi_{m} f_{N,m}(T)\d V_{m}-\int_{\R^{3m}}\vphi_{m}f_{m}(T)\d V_{m}\right|\leq\left\|\vphi_{m}\right\|_{L^\infty\left((\B_R)^m\right)}\big\|f_{N,m}(T)-f_{m}(T)\big\|_{L^1\left((\B_R)^m\right)},
$$
which goes to $0$ when $N\to\infty$ (up to a subsequence) thanks to Proposition \ref{prop strong convergence}. Secondly, the first two terms on the right-hand side of \eqref{eq weak BBGKY} have the following estimates due to the estimates \eqref{ineq fNm A} and \eqref{ineq fNm B} in Lemma \ref{lemma uniform bounds with time},
$$
\begin{aligned}
\&\bigg|\frac{1}{N}\sum_{i\neq j}^m\int_0^T\int_{\R^{3m}}A(v^i-v^j):(\nabla_{v^iv^i}^2\vphi_{m}-\nabla_{v^iv^j}^2\vphi_{m})f_{N,m}(t)\d V_{m}\d t\bigg|\leq C(\vphi_m, f^0)\frac{m^2T}{N},
\end{aligned}
$$
and
$$
\begin{aligned}
\&
\bigg|\frac{1}{N}\sum_{i\neq j}^m\int_0^T\int_{\R^{3m}}B(v^i-v^j)\cdot (\nabla_{v^i}\vphi_{m}-\nabla_{v^j}\vphi_{m}) f_{N,m}(t)\d V_{m}\d t\bigg|\leq C(\vphi_m, f^0)\frac{m^2T}{N},   
\end{aligned}
$$
which converge to $0$ when $N\to\infty$.
The rest of the section will focus on demonstrating the convergence of the last two terms of \eqref{eq weak BBGKY}. In order to pass to the limit of the term involving the vector $B$, we decompose the distance between it and its large $N$ limit as follows:
\begin{align}\label{pass to the limit B}
\notag\&\bigg|\frac{N-m}{N}\sum_{i=1}^m \int_0^T\int_{\R^{3(m+1)}}  B(v^i-v^{m+1})\cdot \left(\nabla_{v^{i}}\vphi_{m}(V_m)-\nabla_{v^{i}}\vphi_{m}(V_m^{i,m+1})\right)f_{N,m+1}\d V_{m+1}\d t\\\notag
\&\qquad-\sum_{i=1}^m \int_0^T\int_{\R^{3(m+1)}}  B(v^i-v^{m+1})\cdot \left(\nabla_{v^{i}}\vphi_{m}(V_m)-\nabla_{v^{i}}\vphi_{m}(V_m^{i,m+1})\right)f_{m+1}\d V_{m+1}\d t\bigg|\\\notag
\leq \& \frac{m^2}{N}\int_0^T\bigg|\int_{\R^{3(m+1)}}  B(v^1-v^{m+1})\cdot \left(\nabla_{v^{1}}\vphi_{m}(V_m)-\nabla_{v^{1}}\vphi_{m}(V_m^{1,m+1})\right)f_{N,m+1}\d V_{m+1}\bigg|\d t\\\notag
\&+m \bigg|\int_0^T\int_{\R^{3(m+1)}}  B(v^1-v^{m+1})\cdot \left(\nabla_{v^{1}}\vphi_{m}(V_m)-\nabla_{v^{1}}\vphi_{m}(V_m^{1,m+1})\right)\left(f_{N,m+1}-f_{m+1}\right)\d V_{m+1}\d t\bigg|\\
\leq\&
\frac{m^2}{N}\left|\int_0^T\textbf{I}(t,N)\d t\right|+m \left|\int_0^T\textbf{J}(t,N)\d t\right|,
\end{align}
where we define
$$
\textbf{I}(t,N):=\int_{\R^{3(m+1)}}  B(v^1-v^{m+1})\cdot \left(\nabla_{v^{1}}\vphi_{m}(V_m)-\nabla_{v^{1}}\vphi_{m}(V_m^{1,m+1})\right)f_{N,m+1}(t)\d V_{m+1},
$$
and
$$
\textbf{J}(t,N):=\int_{\R^{3(m+1)}}  B(v^1-v^{m+1})\cdot \left(\nabla_{v^{1}}\vphi_{m}(V_m)-\nabla_{v^{1}}\vphi_{m}(V_m^{1,m+1})\right)\left(f_{N,m+1}(t)-f_{m+1}(t)\right)\d V_{m+1}.
$$
By the uniform bound \eqref{ineq fNm+1 B} and in Lemma \ref{lemma uniform bounds with time}, the first term asymptotically vanishes such as
$$\frac{m^2}{N}\left|\int_0^T\textbf{I}(t,N)\d t\right|\to 0\quad\text{as}\quad N\to\infty.$$
The integral  $\mathbf{J}$  is decomposed as follows:
$$
\begin{aligned}
\textbf{J}(t,N)=\&\int_{\R^{3(m+1)}} \Big(\chi_{\delta,in}(|v^1-v^{m+1}|)+\chi_{\delta,out}(|v^1-v^{m+1}|)+\chi_{\delta,mid}(|v^1-v^{m+1}|)\Big) \\\&\qquad\qquad\qquad B(v^1-v^{m+1})\cdot \left(\nabla_{v^{1}}\vphi_{m}(V_m)-\nabla_{v^{1}}\vphi_{m}(V_m^{1,m+1})\right)\left(f_{N,m+1}-f_{m+1}\right)\d V_{m+1}\\
=:\&\textbf{J}_{in}(t,\delta,N)+\textbf{J}_{out}(t,\delta,N)+\textbf{J}_{mid}(t,\delta,N),
\end{aligned}
$$
where we introduce smooth functions $\chi_{\delta,in}\,,\,\chi_{\delta,out}\,,\,\chi_{\delta,mid}\in C^\infty([0,\infty))$ that satisfy the following conditions for some small $\delta>0$:
$$
\chi_{\delta,in}(r)=\left\{\begin{array}{cc}
    1 & \quad\mbox{ for } 0\leq r\leq \delta/2 \\
     
      0 & \mbox{ for }  r\geq \delta 
\end{array}\right., \qquad\chi_{\delta,out}(r)=\left\{\begin{array}{cc}
    0 & \quad\mbox{ for } 0\leq r\leq  1/\delta \\
     
      1 & \mbox{ for } r\geq 2/\delta 
\end{array}\right.,
$$
and
$$
\displaystyle\chi_{\delta,mid}(r)=\left\{\begin{array}{cc}
    0 & \quad\mbox{ for } 0\leq r\leq  \delta/2 \\
     1 & \quad\mbox{ for } \delta\leq r< 1/\delta \\
      0 & \mbox{ for } r\geq 2/\delta 
\end{array}\right..
$$
And these functions form a partition of the unity, ensuring that
$$
\chi_{\delta,in}(r)+\chi_{\delta,out}(r)+\chi_{\delta,mid}(r)=1,\quad \forall r\in[0,\infty).
$$
For the first term $\textbf{J}_{in}$, we have the estimate,
$$
\begin{aligned}
\left|\textbf{J}_{in}(t,\delta,N) \right|\leq\&2\|\chi_{\delta,in}\|_{L^\infty}\int_{|v^1-v^{m+1}|\leq \delta}  \frac{\big|\nabla_{v^{1}}\vphi_{m}(V_m)-\nabla_{v^{1}}\vphi_{m}(V_m^{1,m+1})\big|}{|v^1-v^{m+1}|^2} \Big|f_{N,m+1}+f_m\Big|\d V_{m+1} \\\leq\&2\sqrt{\delta}\int_{|v^1-v^{m+1}|\leq \delta}  \frac{\big|\nabla_{v^{1}}\vphi_{m}(V_m)-\nabla_{v^{1}}\vphi_{m}(V_m^{1,m+1})\big|}{|v^1-v^{m+1}|} \frac{|f_{N,m+1}+f_m|}{|v^1-v^{m+1}|^{\frac{3}{2}}}\d V_{m+1} \\\leq\&2\sqrt{\delta}\left\|\nabla_{v^1v^1}^2\vphi_m\right\|_{L^\infty}\bigg(\int_{(\B_{R+\delta})^2}\frac{f_{N,2}}{|v^1-v^{m+1}|^{\frac{3}{2}} }\d v^1\d v^{m+1}\\&\qquad\qquad\qquad\qquad\qquad\qquad+\int_{(\B_{R+\delta})^2}\frac{f_{2}}{|v^1-v^{m+1}|^{\frac{3}{2}} }\d v^1\d v^{m+1} \bigg),
\end{aligned}
$$
where $f_2$ here is understood as the second-marginal of $f_{m+1}$. Notice that in the rest we will not specify anymore the norms of the cut-off functions since they are always taken in $L^\infty$-norm and they are equal to 1.
We perform the change of variables similar to \eqref{change of variables} such as
$$
(v^1,v^{m+1})\mapsto\frac{1}{\sqrt{2}}(v^1-v^{m+1},v^1+v^{m+1})=:(w^1,w^2),
$$
and have the similar estimate as \eqref{ineq first marginal}
$$
\begin{aligned}
\int_{(\B_{R+\delta})^2}\&\frac{f_{N,2}}{|v^1-v^{m+1}|^{\frac{3}{2}} }\d v^1\d v^{m+1}\leq \frac{1}{2^{3/4}}\int_{(\B_{2R})^2}\frac{\tilde{f}_{N,2}(w^1,w^2)}{|w^1|^{\frac{3}{2}}} \d w^1\d w^2\leq \frac{1}{2^{3/4}}\int_{\B_{2R}}\frac{\tilde{f}_{N,1}(w^1)}{|w^1|^{\frac{3}{2}} } \d w^1\\
\leq\& \frac{1}{2^{3/4}} \left(\int_{\B_{2R}}\frac{1}{|w^1|^{\frac{9}{4}}}\d w^1\right)^{\frac{2}{3}}\left\|\tilde{f}_{N,1}\right\|_{L^3(\B_{2R})}\leq C_p\left\|\tilde{f}_{N,1}\right\|_{L^3(\B_{2R})}.
\end{aligned}
$$
The compact embedding imply that $H^1(\B_{2R})\hookrightarrow L^6(\B_{2R})$, 
$$
\left\|\tilde{f}_{N,1}\right\|_{L^3(\B_{R'})}\leq C\left\|\nabla\sqrt{\tilde{f}_{N,1}}\right\|_{L^2(\B_{R'})}^2= \frac{C}{4}\I_1(\tilde{f}_{N,1}),
$$
and Lemma \ref{lemma subadditivity} together with the uniform bound \eqref{ineq fNm fisher} imply that for any $t\in[ 0,T]$,
$$
\I_1(\tilde{f}_{N,1})\leq \I_2(\tilde{f}_{N,2})=\I_2(f_{N,2})\leq \I_N(f_{N})<\I_1(f^0).
$$
Therefore, we get the uniform-in-$N$ bound such as
\begin{equation}\label{ineq supt fN2}
\sup_{t\in[0,T]}\int_{(\B_{R+\delta})^2}\frac{f_{N,2}(t,v^1,v^{m+1})}{|v^1-v^{m+1}|^{\frac{3}{2}} }\d v^1\d v^{m+1}<C,
\end{equation}
where the constant $C$ does not dependent on $N$. Similar bound holds for $f_2$ as well
\begin{equation}\label{ineq supt f2}
\sup_{t\in[0,T]}\int_{(\B_{R+\delta})^2}\frac{f_{2}(t,v^1,v^{m+1})}{|v^1-v^{m+1}|^{\frac{3}{2}} }\d v^1\d v^{m+1}<C.
\end{equation}
Substituting the bound \eqref{ineq supt fN2} and \eqref{ineq supt f2} into $\int_0^T|\textbf{J}_{in}(t,\delta,N)|\d t$, we obtain the uniform-in-$N$ bound
\begin{equation}\label{ineq In}
\int_0^T|\textbf{J}_{in}(t,\delta,N)|\d t\leq T\sup_{t\in[0,T]}|\textbf{J}_{in}(t,\delta,N)|\leq C(f^0,\vphi_m)T\sqrt{\delta}. 
\end{equation}
For the second term $\textbf{J}_{out}$, we have the uniform-in-$N$ estimate,
\begin{align}\label{ineq Out}
\notag\int_0^T\big|\textbf{J}_{out}(t,\delta,N)\& \big|\d t\\\leq\notag\&2\int_0^T\int_{|v^1-v^{m+1}|\geq \frac{1}{\delta}}  \frac{\left|\nabla_{v^{1}}\vphi_{m}(V_m)-\nabla_{v^{1}}\vphi_{m}(V_m^{1,m+1})\right|}{|v^1-v^{m+1}|^2} \left(f_{N,m+1}+f_{m+1}\right)\d V_{m+1}\d t \\\leq\&\notag4\delta^2\|\nabla_{v^1}\vphi_m\|_{L^\infty(\R^{3m})}\int_0^T\int_{|v^1-v^{m+1}|\geq \frac{1}{\delta}} \left(f_{N,m+1}+f_{m+1}\right)\d V_{m+1}\d t\\\leq\&8T\|\nabla_{v^1}\vphi_m\|_{L^\infty(\R^{3m})}\delta^2.
\end{align}
For the third term $\textbf{J}_{mid}$, we let the test function on $\R^{3(m+1)}$ be
$$
\psi_{m+1}:=\chi_{\delta,mid}(|v^1-v^{m+1}|)\,B(|v^1-v^{m+1}|)\cdot \left(\nabla_{v^{1}}\vphi_{m}(V_m)-\nabla_{v^{1}}\vphi_{m}(V_m^{1,m+1})\right),
$$
which belongs to $W^{3,\infty}(\R^{3(m+1)})$ since $\vphi_m\in W^{4,\infty}(\R^{3m})$, and it is also compactly supported by construction. Using Proposition \ref{prop strong convergence}, we know that
$$
t\mapsto\int_{\R^{3(m+1)}}\psi_{m+1}f_{N,m+1}(t)\d V_{m+1}\quad\text{and}\quad t\mapsto\int_{\R^{3(m+1)}}\psi_{m+1}f_{m+1}(t)\d V_{m+1},
$$
are both continuous on $[0,T]$, and thus there exists a $r\in[0,T]$ such that 
$$
\begin{aligned}
\left|\int_0^T\textbf{J}_{mid}(t,\delta,N) \d t\right|=\&\left|\int_0^T\int_{\R^{3(m+1)}}\psi_{m+1}\left(f_{N,m+1}(t)-f_{m+1}(t)\right)\d V_{m+1}\d t\right|\\
=\& T \left|\int_{\R^{3(m+1)}}\psi_{m+1}\left(f_{N,m+1}(r)-f_{m+1}(r)\right)\d V_{m+1}\right|,
\end{aligned}
$$
due to the mean-value theorem. Proposition \ref{prop strong convergence} yields that there exists a subset of the dense set $\{\tau_k\}_{k\in\N}\subset[0,T]$ such that $\lim_{k'\to\infty}\tau_{k'}=r$ and $f_{N_\ell,m+1}$ converges to $f_{m+1}$ uniformly on $\{\tau_{k'}\}_{k'\in\N}$. Then, by \eqref{bound equicontinuity N}, there exists some constant $C$ independent of $N$, $r$ and $\tau_{k'}$ such that
$$
\begin{aligned}
\bigg|\int_{\R^{3(m+1)}}\&\psi_{m+1}\left(f_{N,m+1}(r)-f_{m+1}(r)\right)\d V_{m+1}\bigg| \\
\leq\& \left|\int_{\R^{3(m+1)}}\psi_{m+1}\left(f_{N,m+1}(r)-f_{N,m+1}(\tau_{k'})\right)\d V_{m+1}\right|\\&+\left|\int_{\R^{3(m+1)}}\psi_{m+1}\left(f_{N,m+1}(\tau_{k'})-f_{m+1}(\tau_{k'})\right)\d V_{m+1}\right|\\&+\left|\int_{\R^{3(m+1)}}\psi_{m+1}\left(f_{m+1}(\tau_{k'})-f_{m+1}(r)\right)\d V_{m+1}\right|\\
\leq\& C|r-\tau_{k'}|+\left\|\psi_{m+1}\right\|_{L^\infty(\R^{3(m+1)})}\int_{supp\,\psi_{m+1}}\left|f_{N,m+1}(\tau_{k'})-f_{m+1}(\tau_{k'})\right|\d V_{m+1},
\end{aligned}
$$
which implies that for some $\tau_{k'}$ closed enough to $r$ such that $|r-\tau_{k'}|<\delta$, 
\begin{equation}\label{ineq Mid}
\left|\int_0^T\textbf{J}_{mid}(t,\delta,N) \d t\right|\leq CT\delta+T\left\|\psi_{m+1}\right\|_{L^\infty(\R^{3(m+1)})}\int_{supp\,\psi_{m+1}}\left|f_{N,m+1}(\tau_{k'})-f_{m+1}(\tau_{k'})\right|\d V_{m+1}.
\end{equation}
Combining \eqref{ineq In}, \eqref{ineq Out} and \eqref{ineq Mid}, it has the estimate for the second term of \eqref{pass to the limit B}
$$
\begin{aligned}
m \bigg|\int_0^T\textbf{J}(t,N)\&\d t\bigg|=m\left|\int_0^T\left(\textbf{J}_{in}(t,\delta,N)+\textbf{J}_{out}(t,\delta,N)+\textbf{J}_{mid}(t,\delta,N)\right)\d t\right| \\
\leq \& C(f^0,\vphi_m)T(\sqrt{\delta}+\delta+\delta^2)+C(\vphi_m,\delta)T\int_{supp\psi_{m+1}}\left|f_{N,m+1}(\tau_{k'})-f_{m+1}(\tau_{k'})\right|\d V_{m+1},
\end{aligned}    
$$
where the last integral converges uniformly in the choice of $\tau_{k'}$ as $N_\ell\to\infty$ by Proposition \ref{prop strong convergence}; then this whole term converges to $0$ by the arbitrariness of $\delta$. Therefore, \eqref{pass to the limit B} converges to $0$ as $N_\ell\to\infty$.

\medskip

Now, we move to the estimate involving the matrix $A$, where the similar decomposition of the distance between it and its large $N$ limit can be performed  
\begin{equation}\label{pass to the limit A}
\begin{aligned}
\bigg|\frac{N-m}{N}\&\sum_{i=1}^m \int_0^T\int_{\R^{3(m+1)}}  A(v^i-v^{m+1}): \nabla^2_{v^{i}v^{i}}\vphi_{m}(V_m)f_{N,m+1}\d V_{m+1}\d t\\
\&\qquad\qquad-\sum_{i=1}^m \int_0^T\int_{\R^{3(m+1)}}  A(v^i-v^{m+1}): \nabla^2_{v^{i}v^{i}}\vphi_{m}(V_m)f_{m+1}\d V_{m+1}\d t\bigg|\\
\leq \& \frac{m^2}{N} \int_0^T\bigg|\int_{\R^{3(m+1)}}  A(v^1-v^{m+1}): \nabla^2_{v^{1}v^{1}}\vphi_{m}(V_m)f_{N,m+1}\d V_{m+1}\bigg|\d t\\
\&+m \bigg|\int_0^T\int_{\R^{3(m+1)}}  A(v^1-v^{m+1}): \nabla^2_{v^{1}v^{1}}\vphi_{m}(V_m)\left(f_{N,m+1}-f_{m+1}\right)\d V_{m+1}\d t\bigg|\\
\leq\& \frac{m^2}{N}\left|\int_0^T\textbf{I}'(t,N)\d t\right|+m \left|\int_0^T\textbf{J}'(t,N)\d t\right|,
\end{aligned}
\end{equation}
where we define
$$
\textbf{I}'(t,N):=\int_{\R^{3(m+1)}}  A(v^1-v^{m+1}): \nabla^2_{v^{1}v^{1}}\vphi_{m}(V_m)f_{m+1}\d V_{m+1},
$$
and
$$
\textbf{J}'(t,N):=\int_{\R^{3(m+1)}}  A(v^1-v^{m+1}): \nabla^2_{v^{1}v^{1}}\vphi_{m}(V_m)\left(f_{N,m+1}-f_{m+1}\right)\d V_{m+1}.
$$
By the uniform bound \eqref{ineq fNm+1 A} and  in Lemma \ref{lemma uniform bounds with time}, the first term asymptotically vanishes such as
$$\frac{m^2}{N}\left|\int_0^T\textbf{I}'(t,N)\big|\d t\right|\to 0\quad\text{as}\quad N\to\infty.$$
By using the same partition of the unity $\chi_{\delta,in}\,,\,\chi_{\delta,out}\,,\,\chi_{\delta,mid}\in C^\infty([0,\infty))$ for some small $\delta>0$, we have the following decomposition,
$$
\begin{aligned}
\textbf{J}'(t,N)=\&\int_{\R^{3(m+1)}} \Big(\chi_{\delta,in}(|v^1-v^{m+1}|)+\chi_{\delta,out}(|v^1-v^{m+1}|)+\chi_{\delta,mid}(|v^1-v^{m+1}|)\Big) \\\&\qquad\qquad\qquad A(v^1-v^{m+1}): \nabla^2_{v^{1}v^{1}}\vphi_{m}(V_m)\left(f_{N,m+1}-f_{m+1}\right)\d V_{m+1}\\
=:\&\textbf{J}'_{in}(t,\delta,N)+\textbf{J}'_{out}(t,\delta,N)+\textbf{J}'_{mid}(t,\delta,N).
\end{aligned}
$$ 
We then have the uniform-in-$N$ estimates parallel to \eqref{ineq In} and \eqref{ineq Out} such as
\begin{equation}\label{ineq In'}
\begin{aligned}
\int_0^T|\textbf{J}'_{in}(t,\delta,N)|\d t\leq\&2\sqrt{\delta}T\left\|\nabla_{v^1v^1}^2\vphi_m\right\|_{L^\infty}\sup_{t\in[0,T]}\bigg(\int_{(\B_{R+\delta})^2}\frac{f_{N,2}}{|v^1-v^{m+1}|^{\frac{3}{2}} }\d v^1\d v^{m+1}\\\&\qquad\qquad+\int_{(\B_{R+\delta})^2}\frac{f_{2}}{|v^1-v^{m+1}|^{\frac{3}{2}} }\d v^1\d v^{m+1} \bigg)\\
\leq \& C(f^0,\vphi_m)T\sqrt{\delta},
\end{aligned}
\end{equation}
and
\begin{equation}\label{ineq Out'}
\begin{aligned}
\int_0^T\left|\textbf{J}'_{out}(t,\delta,N) \right|\d t\leq\&\int_0^T\int_{|v^1-v^{m+1}|\geq \frac{1}{\delta}}  \frac{\left|\nabla^2_{v^{1}v^{1}}\vphi_{m}(V_m)\right|}{|v^1-v^{m+1}|} \left(f_{N,m+1}+f_{m+1}\right)\d V_{m+1}\d t \\\leq&\delta \|\nabla^2_{v^1v^1}\vphi_m\|_{L^\infty(\R^{3m})}\int_0^T\int_{|v^1-v^{m+1}|\geq \frac{1}{\delta}} \left(f_{N,m+1}+f_{m+1}\right)\d V_{m+1}\d t\\\leq&C(\vphi_m)T\delta.
\end{aligned}
\end{equation}
We further let the test function
$$
\psi_{m+1}:=\chi_{\delta,mid}(|v^1-v^{m+1}|)A(v^1-v^{m+1}): \nabla^2_{v^{1}v^{1}}\vphi_{m}(V_m),
$$
which belongs to $W^{2,\infty}(\R^{3(m+1)})$ since $\vphi_m\in W^{4,\infty}(\R^{3m})$ and it is compactly supported by construction. By the same argument as \eqref{ineq Mid}, there exists some time $\tau_k$ in the dense set and some constant $C$ independent of $N$ and $\tau_k$ such that
\begin{equation}\label{ineq Mid'}
\begin{aligned}
\bigg|\int_0^T\&\textbf{J}'_{mid}(t,\delta,N) \d t\bigg|\\\leq\& C(\psi_{m+1})T\delta+T\left\|\psi_{m+1}\right\|_{L^\infty(\R^{3(m+1)})}\int_{supp\psi_{m+1}}\left|f_{N,m+1}(\tau_k)-f_{m+1}(\tau_k)\right|\d V_{m+1},    
\end{aligned}
\end{equation}
where the convergence of the last term is uniform in $\tau_{k}$. Therefore, collecting \eqref{ineq In'}, \eqref{ineq Out'} and \eqref{ineq Mid'}, we arrive to the final estimate for the second term of \eqref{pass to the limit A}
$$
\begin{aligned}
m \left|\int_0^T\textbf{J}'(t,N)\d t\right|=\&m\left|\int_0^T\left(\textbf{J}'_{in}(t,\delta,N)+\textbf{J}'_{out}(t,\delta,N)+\textbf{J}'_{mid}(t,\delta,N)\right)\d t\right| \\
\leq \& C(f^0,\vphi_m)T(\sqrt{\delta}+\delta)+C(\vphi_m,\delta)T\int_{supp\psi_{m+1}}\left|f_{N,m+1}(\tau_k)-f_{m+1}(\tau_k)\right|\d V_{m+1},
\end{aligned}    
$$
which goes to $0$ when $N_\ell\to\infty$ by Proposition \ref{prop strong convergence} and the arbitrariness of $\delta$. Thus, we show that \eqref{pass to the limit A} converges to $0$.

\medskip

We conclude Proposition \ref{prop pass to the limit} and complete the proof of Theorem \ref{thm main}.

\section*{Acknowledgments}
JAC and SG were supported by the Advanced Grant Nonlocal-CPD (Nonlocal PDEs for Complex Particle Dynamics: Phase Transitions, Patterns and Synchronization) of the European Research Council Executive Agency (ERC) under the European Union Horizon 2020 research and innovation programme (grant agreement No. 883363) and  partially supported by the EPSRC EP/V051121/1. JAC was partially supported by the ``Maria de Maeztu'' Excellence Unit IMAG, reference CEX2020-001105-M, funded
by MCIN/AEI/10.13039/501100011033/.

\bibliographystyle{abbrv}
\bibliography{ref}

\begin{thebibliography}{10}

\bibitem{bresch2022new}
D.~Bresch, P.-E. Jabin, and J.~Soler.
\newblock A new approach to the mean-field limit of {V}lasov-{F}okker-{P}lanck equations.
\newblock {\em Analysis and PDE, to appear}, 2025.

\bibitem{carrapatoso2016propagation}
K.~Carrapatoso.
\newblock Propagation of chaos for the spatially homogeneous {L}andau equation for {M}axwellian molecules.
\newblock {\em Kinetic and Related Models}, 9(1):1--49, 2016.

\bibitem{carrillo2024relative}
J.~A. Carrillo, X.~Feng, S.~Guo, P.-E. Jabin, and Z.~Wang.
\newblock Relative entropy method for particle approximation of the {L}andau equation for {M}axwellian molecules.
\newblock {\em arXiv preprint arXiv:2408.15035}, 2024.

\bibitem{carrillo2024mean}
J.~A. Carrillo, S.~Guo, and P.-E. Jabin.
\newblock Mean-field derivation of {L}andau-like equations.
\newblock {\em Applied Mathematics Letters}, page 109195, 2024.

\bibitem{fontbona2009measurability}
J.~Fontbona, H.~Gu{\'e}rin, and S.~M{\'e}l{\'e}ard.
\newblock Measurability of optimal transportation and convergence rate for {L}andau type interacting particle systems.
\newblock {\em Probability theory and related fields}, 143(3):329--351, 2009.

\bibitem{fournier2009particle}
N.~Fournier.
\newblock Particle approximation of some {L}andau equations.
\newblock {\em Kinetic and Related Models}, 2(3):451--464, 2009.

\bibitem{fournier2017kac}
N.~Fournier and A.~Guillin.
\newblock From a {K}ac-like particle system to the {L}andau equation for hard potentials and {M}axwell molecules.
\newblock In {\em Annales Scientifiques de l'{\'E}cole Normale Sup{\'e}rieure}, volume~50, pages 157--199, 2017.

\bibitem{fournier2016propagation}
N.~Fournier and M.~Hauray.
\newblock Propagation of chaos for the {L}andau equation with moderately soft potentials.
\newblock {\em The Annals of Probability}, 44(6):3581--3660, 2016.

\bibitem{fournier2014propagation}
N.~Fournier, M.~Hauray, and S.~Mischler.
\newblock Propagation of chaos for the 2d viscous vortex model.
\newblock {\em Journal of the European Mathematical Society}, 16(7):1423--1466, 2014.

\bibitem{guillen2025landau}
N.~Guillen and L.~Silvestre.
\newblock The {L}andau equation does not blow up.
\newblock {\em Acta Mathematica, to appear}, 2025.

\bibitem{hauray2014kac}
M.~Hauray and S.~Mischler.
\newblock On {K}ac's chaos and related problems.
\newblock {\em Journal of Functional Analysis}, 266(10):6055--6157, 2014.

\bibitem{imbert2024monotonicity}
C.~Imbert, L.~Silvestre, and C.~Villani.
\newblock On the monotonicity of the {F}isher information for the {B}oltzmann equation.
\newblock {\em arXiv preprint arXiv:2409.01183}, 2024.

\bibitem{ji2024bounds}
S.~Ji.
\newblock Bounds for the optimal constant of the {B}akry-Émery {$\Gamma_2 $} criterion inequality on {$\mathbb{R}P^{d-1}$}.
\newblock {\em arXiv preprint arXiv:2408.13954}, 2024.

\bibitem{kac1956foundations}
M.~Kac.
\newblock Foundations of kinetic theory.
\newblock In {\em Proceedings of the {T}hird {B}erkeley {S}ymposium on {M}athematical {S}tatistics and {P}robability, 1954--1955, vol. {III}}, pages 171--197. Univ. California Press, Berkeley-Los Angeles, Calif., 1956.

\bibitem{kiessling2004master}
M.~Kiessling and C.~Lancellotti.
\newblock On the master-equation approach to kinetic theory: Linear and nonlinear {F}okker-{P}lanck equations.
\newblock {\em Transport Theory and Statistical Physics}, 33(5-7):379--401, 2004.

\bibitem{landau1936kinetische}
L.~Landau.
\newblock Die kinetische gleichung fuer den fall coulombscher wechselwirkung.
\newblock {\em Phys. Z. Sowjet. 10 (1936)}, 10:163, 1936.

\bibitem{ledoux2022differentials}
M.~Ledoux.
\newblock Differentials of entropy and {F}isher information along heat flow: a brief review of some conjectures.
\newblock {\em Preprint}, 2022.

\bibitem{Miot2011kac}
E.~Miot, M.~Pulvirenti, and C.~Saffirio.
\newblock On the {K}ac model for the {L}andau equation.
\newblock {\em Kinetic and Related Models}, 4(1):333--344, 2011.

\bibitem{mischler2013kac}
S.~Mischler and C.~Mouhot.
\newblock Kac’s program in kinetic theory.
\newblock {\em Inventiones mathematicae}, 193:1--147, 2013.

\bibitem{prigogine1959irreversibleI}
I.~Prigogine and R.~Balescu.
\newblock Irreversible processes in gases {I}. {T}he diagram technique.
\newblock {\em Physica}, 25(1-6):281--301, 1959.

\bibitem{prigogine1959irreversibleII}
I.~Prigogine and R.~Balescu.
\newblock Irreversible processes in gases {II}. {T}he equations of evolution.
\newblock {\em Physica}, 25(1-6):302--323, 1959.

\bibitem{villani1998new}
C.~Villani.
\newblock On a new class of weak solutions to the spatially homogeneous {B}oltzmann and {L}andau equations.
\newblock {\em Archive for rational mechanics and analysis}, 143:273--307, 1998.

\bibitem{villani2025fisher}
C.~Villani.
\newblock Fisher information in kinetic theory.
\newblock {\em arXiv preprint arXiv:2501.00925}, 2025.

\end{thebibliography}

\end{document}